\documentclass[reqno]{amsart}
\usepackage{mathrsfs}
\usepackage{color}
\usepackage{amsmath}
\usepackage{amsfonts}
\usepackage{color}
\usepackage{amssymb}
\usepackage{mathtools}
\usepackage{hyperref}

 \newtheorem{Theorem}{Theorem}[section]
 \newtheorem{Corollary}[Theorem]{Corollary}
 \newtheorem{Lemma}[Theorem]{Lemma}
 \newtheorem{Proposition}[Theorem]{Proposition}

 \newtheorem{Definition}[Theorem]{Definition}
\newtheorem{Question}[Theorem]{Question}
 \newtheorem{Conjecture}[Theorem]{Conjecture}

 \newtheorem{Remark}[Theorem]{Remark}

 \numberwithin{equation}{section}

\begin{document}

\title[Zhou valuations and jumping numbers]
 {Zhou valuations and jumping numbers}
 
 \author{Shijie Bao}
\address{Shijie Bao: Institute of Mathematics, Academy of Mathematic sand Systems Science, Chinese Academy of Sciences, Beijing 100190, China.}
\email{bsjie@amss.ac.cn}

\author{Qi'an Guan}
\address{Qi'an Guan: School of Mathematical Sciences,
	Peking University, Beijing, 100871, China.}
\email{guanqian@math.pku.edu.cn}

\author{Zheng Yuan}
\address{Zheng Yuan: Institute of Mathematics, Academy of Mathematics
	and Systems Science, Chinese Academy of Sciences, Beijing 100190, China.}
\email{yuanzheng@amss.ac.cn}


\thanks{}


\date{\today}

\dedicatory{}

\commby{}
%%% ----------------------------------------------------------------------
\begin{abstract}
	In this article,
	we prove that for any Zhou valuation $\nu$, there exists a graded sequence of ideals $\mathfrak{a}_{\bullet}$ and a nonzero ideal $\mathfrak{q}$ such that $\nu$ $\mathscr{A}-$computes the jumping number $\mathrm{lct}^{\mathfrak{q}}(\mathfrak{a}_{\bullet})$, and that for the subadditive sequence $\mathfrak{b}^{\varphi}_{\bullet}$ related to a plurisubharmonic function $\varphi$, there exists a Zhou valuation which $\mathscr{A}-$computes $\mathrm{lct}^{\mathfrak{q}}(\mathfrak{b}^{\varphi}_{\bullet})$, where the  ``$\mathscr{A}-$compute'' coincides with the ``compute'' in Jonsson-Musta\c{t}\u{a}'s Conjecture when the Zhou valuation $\nu$ is quasimonomial. There are also some results obtained for Zhou valuations, including a characterization for a valuation being a Zhou valuation, and a denseness property of the cone of Zhou valuations.
\end{abstract}

%%% ----------------------------------------------------------------------

\maketitle
\setcounter{tocdepth}{1}
\tableofcontents
%%% ----------------------------------------------------------------------

\section{Introduction}

\subsection{Background}
Let $u$ be a plurisubharmonic function near the origin $o$ of $\mathbb{C}^n$. The Lelong number of $u$ at $o$ was defined as
$$\nu(u,o):=\sup\big\{c\ge0\colon u(z)\le c\log|z|+O(1)\text{ near }o\big\},$$
which gives some information on the asymptotic behavior of $u$ near $o$ (see \cite{demailly-book,demailly2010,Rash06,BFJ08}). In \cite{Rash06}, Rashkovskii introduced a generalization of the Lelong number:

\emph{Let $\varphi$ be a maximal weight with isolated singularity $o$ (see \cite{Rash06,BFJ08}), and let $u$ be a plurisubharmonic function near  $o$. The relative type of $u$ relative to $\varphi$ was defined as: $\sigma(u,\varphi):=\sup\{c\ge0:u\le c\varphi+O(1)\text{ near }o\}.$}

Let $\psi$ be a plurisubharmonic function on a complex manifold. The multiplier ideal sheaf $\mathcal{I}(\psi)$ (see \cite{Nadel90,DEL00,D-K01,demailly-note2000}) was defined as  the sheaf of germs of holomorphic functions $f$ such that $|f|^{2}e^{-2\psi}$ is local integrable. In \cite{BFJ08}, Boucksom-Favre-Jonsson showed that the relative types to all tame maximal weights characterize the multiplier ideal sheaves of plurisubharmonic functions.

Based on the solution of the strong openness conjecture: $\mathcal{I}(\psi)=\mathcal{I}_+(\psi):=\cup_{p>1}\mathcal{I}(p\psi)$ (see \cite{demailly-note2000,GZopen-c}), Bao-Guan-Mi-Yuan \cite{BGMY-valuation} obtained a class of tame maximal weights $\Phi_{o,\max}$ (called Zhou weights), whose relative types characterize the multiplier ideal sheaves of plurisubharmonic functions and are well suited for the tropical structure of the cone of plurisubharmonic functions:
 $$\sigma(u+v,\Phi_{o,\max})=\sigma(u,\Phi_{o,\max})+\sigma(v,\Phi_{o,\max})$$
and
$$\sigma(\max\{u,v\},\Phi_{o,\max})=\min\{\sigma(u,\Phi_{o,\max}),\sigma(v,\Phi_{o,\max})\}.$$
Thus, the relative types to Zhou weights are valuations on the local ring $\mathcal{O}_o$, denoted by: 
$$\nu(f):=\sigma(\log|f|,\Phi_{o,\max}), \ \forall (f,o)\in\mathcal{O}_o,$$
which are called by Zhou valuations \cite{BGMY-valuation}.

Recall that a valuation on $\mathcal{O}_o$ is a non-constant map $\nu:\mathcal{O}_o^*\rightarrow\mathbb{R}_{\ge0}$ satisfying the following:

$(1)$ $\nu(fg)=\nu(f)+\nu(g);$

$(2)$ $\nu(f+g)\ge\min\{\nu(f),\nu(g)\};$

$(3)$ $\nu(c)=0$, where $c\not=0$ is a constant function.

\noindent Here $\mathcal{O}^*_o:=\mathcal{O}_o\setminus\{0\}$. We set $\nu(0)=+\infty$ for any valuation $\nu$. Zhou valuations characterize the division relations of $\mathcal{O}_o$ (\cite{BGMY-valuation}). 

Let $\mathfrak{q}$ be a nonzero ideal of $\mathcal{O}_o$, and let $\varphi$ be a plurisubharmonic function near $o$. Recall that the jumping number (see \cite{JM12,JM14})
$$c_o^{\mathfrak{q}}(\varphi):=\sup\{c\ge0:|\mathfrak{q}|^2e^{-2c\varphi}\text{ is integrable near }o\},$$
where $|\mathfrak{q}|^2:=\sum_{1\le j\le k}|{h}_j|^2$ and $\mathfrak{q}$ is generated by $({h}_1,o),\ldots,({h}_k,o)$.
For any local Zhou weight $\Phi_{o,\max}$ and its corresponding Zhou valuation $\nu$,  $c_o^{f}(\Phi_{o,\max})$ and $\nu(f)$ for all $(f,o)\in\mathcal{O}_o$ are linearly controlled by each other \cite{BGMY-valuation} (see also Lemma \ref{thm:valu-jump}).

Let $\mathfrak{a}_{\bullet}$ be a graded sequence of ideals of $\mathcal{O}_o$: it is a sequence $(\mathfrak{a}_m)_{m\ge1}$ of ideals of $\mathcal{O}_o$ such that $\mathfrak{a}_p\cdot \mathfrak{a}_q\subset \mathfrak{a}_{p+q}$ for any positive integers $p$ and $q$. Let $\nu$ on $\mathcal{O}_o$ be a valuation, and recall that  (see \cite{JM12,JM14})
$$\nu(\mathfrak{a}_{\bullet}):=\inf\frac{\nu(\mathfrak{a}_m)}{m}=\lim_{m\rightarrow+\infty}\frac{\nu(\mathfrak{a}_m)}{m},$$
where $\nu(\mathfrak{c}):=\inf\{\nu(f):(f,o)\in\mathfrak{c}\}$ for any ideal $\mathfrak{c}$ of $\mathcal{O}_o$.
Recall that   (see \cite{JM12,JM14})
$$\mathrm{lct}^{\mathfrak{q}}(\mathfrak{a}_{\bullet}):=\lim_{m\rightarrow+\infty}m\cdot\mathrm{lct}^{\mathfrak{q}}(\mathfrak{a}_{m})=\sup_m m\cdot\mathrm{lct}^{\mathfrak{q}}(\mathfrak{a}_{m})\in\mathbb{R}_{\ge0}\cup\{+\infty\},$$
where the jumping number of $\mathfrak{a}_m$ with respect to $\mathfrak{q}$
$$\mathrm{lct}^{\mathfrak{q}}(\mathfrak{a}_{m}):=c_o^{\mathfrak{q}}(\log|\mathfrak{a}_m|).$$
Note that (see \cite{JM12,JM14})
$$\mathrm{lct}^{\mathfrak{q}}(\mathfrak{a}_{\bullet})=\inf_{\nu}\frac{A(\nu)+\nu(\mathfrak{q})}{\nu(\mathfrak{a}_{\bullet})},$$
where the infimum is over all valuations on $\mathcal{O}_o$ (it is enough to only consider quasimonomial valuations (see \cite{JM12,JM14,xu2019})) and $A(\nu)$ is the log discrepancy of $\nu$ (see \cite{JM12,JM14,xu2019}).

In \cite{JM14}, Jonsson-Musta\c{t}\u{a} reduced the strong openness conjecture to a purely algebraic statement:

\begin{Conjecture}\label{conj:1}
	For any nonzero ideal $\mathfrak{q}$ and any graded sequence $\mathfrak{a}_{\bullet}$ of ideals  of $\mathcal{O}_o$ such that $\mathfrak{m}^p\subset\mathfrak{a}_1$ for the maximal ideal $\mathfrak{m}$ of $\mathcal{O}_o$ and some $p\ge1$, there is a quasimonomial valuation $\nu$ on $\mathcal{O}_o$ that computes $\mathrm{lct}^{\mathfrak{q}}(\mathfrak{a}_{\bullet})$, i.e.,
	$$\mathrm{lct}^{\mathfrak{q}}(\mathfrak{a}_{\bullet})=\frac{A(\nu)+\nu(\mathfrak{q})}{\nu(\mathfrak{a}_{\bullet})}.$$
\end{Conjecture}

A subadditive sequence of ideals $\mathfrak{b}_{\bullet}$ is a sequence  $(\mathfrak{b}_t)_{t\in\mathbb{R}_{>0}}$ of nonzero ideals of $\mathcal{O}_o$ satisfying $\mathfrak{b}_{s+t}\subset\mathfrak{b}_s\cdot\mathfrak{b}_t$ for any $s,t\in\mathbb{R}_{>0}$. For any valuation $\nu$ on $\mathcal{O}_o$ and nonzero ideal $\mathfrak{q}$ of $\mathcal{O}_o$, recall that   (see \cite{JM12,JM14})
$$\nu(\mathfrak{b}_{\bullet}):=\sup_{t}\frac{\nu(\mathfrak{b}_t)}{t}=\lim_{t\rightarrow+\infty}\frac{\nu(\mathfrak{b}_t)}{t}\in\mathbb{R}_{\ge0}\cup\{+\infty\}$$
and 
$$\mathrm{lct}^{\mathfrak{q}}(\mathfrak{b}_{\bullet}):=\inf_{t}t\cdot\mathrm{lct}^{\mathfrak{q}}(\mathfrak{b}_t)=\lim_{t\rightarrow+\infty}t\cdot\mathrm{lct}^{\mathfrak{q}}(\mathfrak{b}_t).$$
We have (see \cite{JM12,JM14})
$$\mathrm{lct}^{\mathfrak{q}}(\mathfrak{b}_{\bullet})=\inf_{\nu}\frac{A(\nu)+\nu(\mathfrak{q})}{\nu(\mathfrak{b}_{\bullet})},$$
where the infimum is over all valuations on $\mathcal{O}_o$ (it is enough to only consider quasimonomial valuations).

The following remark shows that the asymptotic multiplier ideals of $\mathfrak{a}_{\bullet}$ is a subadditive sequence of ideals.
\begin{Remark}
	Let $\mathfrak{a}_{\bullet}$ be a graded sequence of ideals. The  asymptotic multiplier ideals of $\mathfrak{a}_{\bullet}$ was defined by $$\mathfrak{b}_t:=\mathcal{I}\left(\frac{t}{m}\log|\mathfrak{a}_m|\right)_o,$$ where $m$ is divisible enough (depending on $t>0$). It is clear that $\mathfrak{a}_{m}\subset\mathfrak{b}_m$ for any $m$, and it follows from Lemma \ref{l:subaddtive} (the subadditive theorem, see \cite{DEL00}) that $(\mathfrak{b}_t)_{t>0}$ is a subadditive system of ideals. We have $\nu(\mathfrak{a}_{\bullet})=\nu(\mathfrak{b}_{\bullet})$ whenever $A(\nu)<+\infty$ and $$\mathrm{lct}^{\mathfrak{q}}(\mathfrak{a}_{\bullet})=\mathrm{lct}^{\mathfrak{q}}(\mathfrak{b}_{\bullet})$$ 
	for any nonzero ideal $\mathfrak{q}$ (see \cite{JM12}).
\end{Remark}

Jonsson-Musta\c{t}\u{a} also presented an equivalent conjecture \cite{JM14} of Conjecture \ref{conj:1}:
\begin{Conjecture}\label{conj:2}
	Let $\mathfrak{q}$ be a nonzero ideal of $\mathcal{O}_o$,  and let $\mathfrak{b}_{\bullet}$ be a subadditive  sequence  of ideals  of $\mathcal{O}_o$. Assume that $\mathfrak{b}_{\bullet}$ is of controlled growth (i.e.,
$\nu(\mathfrak{b}_{\bullet})\le\frac{\nu(\mathfrak{b}_t)}{t}+\frac{A(\nu)}{t}$
for any $t>0$ and any quasimonomial valuations $\nu$), 
	 and there is a positive integer $p$ such that $\mathfrak{m}^{pj}\subset\mathfrak{b}_j$ for any $j$ and the maximal ideal $\mathfrak{m}$ of $\mathcal{O}_o$, then there is a quasimonomial valuation $\nu$ on $\mathcal{O}_o$ that computes $\mathrm{lct}^{\mathfrak{q}}(\mathfrak{b}_{\bullet})$, i.e.,
	 $$\mathrm{lct}^{\mathfrak{q}}(\mathfrak{b}_{\bullet})=\frac{A(\nu)+\nu(\mathfrak{q})}{\nu(\mathfrak{b}_{\bullet})}.$$
\end{Conjecture}

 The 2-dimensional Jonsson-Musta\c{t}\u{a}'s Conjecture has been proved in \cite{JM12} (some related results seen in \cite{FM05j,JM14}). Recently, Xu \cite{xu2019} proved the Jonsson-Musta\c{t}\u{a}'s Conjecture for the case $\mathfrak{q}=1$, which completed the algebraic approach of solving the  openness conjecture (see \cite{D-K01,FM05j,berndtsson13}).

In this article, we prove that for any Zhou valuation $\nu$, there exists a graded sequence of ideals $\mathfrak{a}_{\bullet}$ and a nonzero ideal $\mathfrak{q}$ that $\nu$ $\mathscr{A}-$computes the jumping number $\mathrm{lct}^{\mathfrak{q}}(\mathfrak{a}_{\bullet})$. For  subadditive sequence $\mathfrak{b}^{\varphi}_{\bullet}$ related to a plurisubharmonic function $\varphi$,  we show that there exists a Zhou valuation that $\mathscr{A}-$computes $\mathrm{lct}^{\mathfrak{q}}(\mathfrak{b}^{\varphi}_{\bullet})$. 
When the Zhou valuation $\nu$ is quasimonomial,  ``$\mathscr{A}-$compute" coincides with  ``compute". Hence, these results provide some support for Jonsson-Musta\c{t}\u{a}'s Conjecture.
In the second part of this article, we give a characterization for a valuation being a Zhou valuation.

\subsection{Zhou valuation $\mathscr{A}-$computing the jumping number}

Recall the definitions of Zhou weights and Zhou numbers (see \cite{BGMY-valuation}). Let $f_{0}=(f_{0,1},\cdots,f_{0,m})$ be a vector, where $f_{0,1},\cdots,f_{0,m}$ are holomorphic functions near $o$. Let $\varphi_{0}$ be a plurisubharmonic function near $o$, such that $|f_{0}|^{2}e^{-2\varphi_{0}}$ is integrable near $o$.

\begin{Definition}[\cite{BGMY-valuation}]
	\label{def:max_relat}
	We call that $\Phi^{f_0,\varphi_0}_{o,\max}$ ($\Phi_{o,\max}$ for short) is a local Zhou weight related to $|f_{0}|^{2}e^{-2\varphi_{0}}$ near $o$,
	if the following three statements hold
	
	(1) $|f_{0}|^{2}e^{-2\varphi_{0}}|z|^{2N_{0}}e^{-2\Phi_{o,\max}}$ is integrable near $o$
	for large enough $N_{0}\gg0$;
	
	(2) $|f_{0}|^{2}e^{-2\varphi_{0}}e^{-2\Phi_{o,\max}}$
	is not integrable near $o$;
	
	(3) for any plurisubharmonic function $\varphi'\geq\Phi_{o,\max}+O(1)$ near $o$
	such that $|f_{0}|^{2}e^{-2\varphi_{0}}e^{-2\varphi'}$
	is not integrable near $o$,
	$\varphi'=\Phi_{o,\max}+O(1)$ holds.
\end{Definition}
For any local Zhou weight $\Phi_{o,\max}$, we have $\Phi_{o,\max}\ge N\log|z|$ near $o$ for large enough $N$ (see \cite{BGMY-valuation}). 
Let $\Phi_{o,\max}$ be a local Zhou weight related to $|f_{0}|^{2}e^{-2\varphi_{0}}$ near $o$, then $\big(1+\sigma(\varphi_{0},\Phi_{o,\max})\big)\Phi_{o,\max}$ is a local Zhou weight related to $|f_{0}|^{2}$ (see \cite{BGMY-valuation}). Thus, we only consider the case $\varphi_0\equiv0$ in this article.

We call a valuation $\nu$ on $\mathcal{O}_o$ is a Zhou valuation related to $|f_0|^2$ if there exists a local Zhou weight $\Phi_{o,\max}$ related $|f_0|^2$ near $o$ such that $$\sigma(\log|f|,\Phi_{o,\max})=\nu(f)$$
for any $(f,o)\in\mathcal{O}_o$.

Let $\mathfrak{a}_{\bullet}$ be a graded sequence of ideals of $\mathcal{O}_o$, and let $\mathfrak{q}$ be a nonzero ideal of $\mathcal{O}_o$. 
Let  $\nu$ on $\mathcal{O}_o$ be a Zhou valuation related to $|\mathfrak{q}_0|^2$, where $\mathfrak{q}_0$ is a nonzero ideal of $\mathcal{O}_o$. 
We define
$$\mathscr{A}(\nu):=\sup_{\mathfrak{q},\mathfrak{a}_{\bullet}}\big(\nu(\mathfrak{a}_{\bullet})\mathrm{lct}^{\mathfrak{q}}(\mathfrak{a}_{\bullet})-\nu(\mathfrak{q})\big),$$
where the supremum is over all nonzero ideals $\mathfrak{q}$ and graded sequences $\mathfrak{a}_{\bullet}$ satisfying $\mathrm{lct}^{\mathfrak{q}}(\mathfrak{a}_{\bullet})<+\infty$. 

As $\mathrm{lct}^{\mathfrak{q}}(\mathfrak{a}_{\bullet})=\inf_{\nu}\frac{A(\nu)+\nu(\mathfrak{q})}{\nu(\mathfrak{a}_{\bullet})}$ (see \cite{JM12}), we have that the log discrepancy $A(\nu)\ge\mathscr{A}(\nu)$ for  Zhou valuation $\nu$.
When the Zhou valuation $\nu$ is quasimonomial (see \cite{JM12,JM14}), $\mathscr{A}(\nu)$  actually coincides with  $A(\nu)$ (by Proposition 2.5 in \cite{JM14}).

\begin{Definition}
A Zhou valuation $\nu$ $\mathscr{A}-$computes $\mathrm{lct}^{\mathfrak{q}}(\mathfrak{a}_{\bullet})$ for a nonzero ideal $\mathfrak{q}$ and a graded sequence of ideals $\mathfrak{a}_{\bullet}$ if $\mathrm{lct}^{\mathfrak{q}}(\mathfrak{a}_{\bullet})=\frac{\mathscr{A}(\nu)+\nu(\mathfrak{q})}{\nu(\mathfrak{a}_{\bullet})}.$
\end{Definition}

 By definitions, a quasimonomial Zhou valuation $\nu$ $\mathscr{A}-$computes $\mathrm{lct}^{\mathfrak{q}}(\mathfrak{a}_{\bullet})$ if and only if $\nu$ computes $\mathrm{lct}^{\mathfrak{q}}(\mathfrak{a}_{\bullet})$.

Let   $\mathfrak{a}^{\nu}_{\bullet}$ be a graded sequence of ideals of $\mathcal{O}_o$ such that $\mathfrak{a}^{\nu}_m=\{(f,o)\in\mathcal{O}_o:v(f)\ge m\}$ for any $m$. The following theorem shows that for any Zhou valuation $\nu$, there exists a graded sequence of ideals $\mathfrak{a}_{\bullet}$ and a nonzero ideal $\mathfrak{q}$ that $\nu$ $\mathscr{A}-$computes the jumping number $\mathrm{lct}^{\mathfrak{q}}(\mathfrak{a}_{\bullet})$.

\begin{Theorem}\label{thm:compute}
	Let $\mathfrak{q}_0$ be a nonzero ideal of $\mathcal{O}_o$, and let $\nu$ is a Zhou valuation related to $|\mathfrak{q}_0|^2$. Then
 $\nu$ $\mathscr{A}-$computes $\mathrm{lct}^{\mathfrak{q}_0}(\mathfrak{a}^{\nu}_{\bullet})$, i.e.,
	$$\mathrm{lct}^{\mathfrak{q}_0}(\mathfrak{a}^{\nu}_{\bullet})=\frac{\mathscr{A}(\nu)+\nu(\mathfrak{q}_0)}{\nu(\mathfrak{a}^{\nu}_{\bullet})}.$$
\end{Theorem}

The following remark gives a relation between  $\mathscr{A}(\nu)$ with Tian functions.

\begin{Remark}\label{r:A(nu)}
	$\mathscr{A}(\nu)=1-\nu(\mathfrak{q}_0)=\sup_{(h,o)\in\mathcal{O}^*_o}(c_o^{h}(\Phi_{o,\max})-\nu(h))$, where $\Phi_{o,\max}$ is the corresponding local Zhou weight of $\nu$ and $\mathcal{O}^*_o:=\{(f,o)\in\mathcal{O}_o:f\not=0\}$. Denote the Tian function (see \cite{BGMY-valuation}) by $$\mathrm{Tn}(t):=\sup\big\{c\ge0:|\mathfrak{q}_0|^{2t}e^{-2c\Phi_{o,\max}}\text{ is integrable near }o\big\}$$
	for $t\in\mathbb{R}$. 
	Then 
	 $\mathrm{Tn}(t)=\nu(\mathfrak{q}_0)t+\mathscr{A}(\nu)$ for any $t\ge1$.
\end{Remark}

Let $\mathfrak{b}_{\bullet}$ be a  subadditive sequence of ideals.
Proposition \ref{p:inf-b} shows that for any subadditive sequence $\mathfrak{b}_{\bullet}$ and any nonzero ideal $\mathfrak{q}$,
$$\mathrm{lct}^{\mathfrak{q}}(\mathfrak{b}_{\bullet})=\inf_{\nu}\frac{\mathscr{A}(\nu)+\nu(\mathfrak{q})}{\nu(\mathfrak{b}_{\bullet})},$$
where the infimum is over all Zhou valuation $\nu$. We call a Zhou valuation $\nu$ $\mathscr{A}-$computes $\mathrm{lct}^{\mathfrak{q}}(\mathfrak{b}_{\bullet})$ if $\mathrm{lct}^{\mathfrak{q}}(\mathfrak{b}_{\bullet})=\frac{\mathscr{A}(\nu)+\nu(\mathfrak{q})}{\nu(\mathfrak{b}_{\bullet})}$.

In the following, we consider a class of subadditive sequences $\mathfrak{b}^{\varphi}_{\bullet}$ related to plurisubharmonic functions $\varphi$ and show that there exist Zhou valuations, which $\mathscr{A}-$compute $\mathrm{lct}^{\mathfrak{q}}(\mathfrak{b}^{\varphi}_{\bullet})$.

Let $\varphi$ be a plurisubharmonic function near $o$. Denote $$\mathfrak{b}_t^{\varphi}:=\mathcal{I}(t\varphi)_o.$$ Then $\mathfrak{b}^{\varphi}_{\bullet}=(\mathfrak{b}^{\varphi}_t)_{t>0}$ is a subadditive sequence of ideals satisfying $\mathfrak{b}^{\varphi}_{s+t}\subset\mathfrak{b}^{\varphi}_s\cdot\mathfrak{b}^{\varphi}_t$ for any $s,t\in\mathbb{R}_{>0}$ by Lemma \ref{l:subaddtive} (the subadditive theorem, see \cite{DEL00}).
Assume that $\varphi\ge N\log|z|+O(1)$ near $o$ for large enough $N$.

\begin{Theorem}
	\label{thm:compute2}
	For any nonzero ideal $\mathfrak{q}$ of $\mathcal{O}_o$,
there exists a Zhou valuation $\nu$ related to $|\mathfrak{q}|^2$, which $\mathscr{A}-$computes $\mathrm{lct}^{\mathfrak{q}}(\mathfrak{b}^{\varphi}_{\bullet})$, i.e., 
$$\mathrm{lct}^{\mathfrak{q}}(\mathfrak{b}^{\varphi}_{\bullet})=\frac{\mathscr{A}(\nu)+\nu(\mathfrak{q})}{\nu(\mathfrak{b}^{\varphi}_{\bullet})}.$$
\end{Theorem}

Due to the above discussions, it is natural to consider coming up with some conditions or examples of Zhou valuations being quasimonomial valuations. Specifically, we can see that \emph{a Zhou valuation associated to a Zhou weight with analytic singularities is also a quasimonomial valuation}, which can be yielded by the following remark.

\begin{Remark}\label{rmk-ana.sing.is.quasimonomial}
	Let $\Psi$ be a plurisubharmonic function near the origin $o$ of $\mathbb{C}^n$ with analytic singularities (i.e. there exists $c\in\mathbb{R}^+$ and $r$ local holomorphic functions $f_1,\ldots,f_r$ near $o$ such that
\[\Psi=c\log\bigl(|f_1|^2+\cdots+|f_r|^2\bigr)+O(1)\]
near $o$). Assume that $o$ is an isolated singularity of $\Psi$. If there exists a valuation $\nu$ on $\mathcal{O}_o$ that coincides with the relative type to $\Psi$, i.e.
\[\nu(h)=\sigma(\log|h|,\Psi), \ \forall (h,o)\in\mathcal{O}_o,\]
then $\nu$ is a divisorial valuation (thus also quasimonomial). 
\end{Remark}

\subsection{A characterization for a valuation being a Zhou valuation}

In a private conversation with the authors, Jonsson posed the following question:
\begin{Question}\label{Q:J}
 Can one characterize the valuations on $\mathcal{O}_o$ that appear as Zhou valuations?
\end{Question}

	Let $\mathfrak{q}_0$ be a nonzero ideal of $\mathcal{O}_o$.
For any valuation $\nu$ on $\mathcal{O}_o$ and positive integer $m$, denote 
 $$\mathfrak{a}_m^{\nu}:=\{(f,o)\in\mathcal{O}_o:\nu(f)\ge m\},$$
which is an ideal of $\mathcal{O}_o$.

We give a characterization for a valuation being a Zhou valuation, which gives an answer to Question \ref{Q:J}.

\begin{Theorem}\label{the:char}
	A valuation $\nu$ on $\mathcal{O}_o$ is a Zhou valuation related to $|\mathfrak{q}_0|^2$ if and only if the following three statements hold:
	
	$(1)$ $\{(f,o)\in\mathcal{O}_o:v(f)>0\}$ is the maximal ideal of $\mathcal{O}_o$;
	
	$(2)$ for any positive integer $m$,  $$\mathfrak{q}_0\not\subset\mathcal{I}\left(\frac{1}{m}\log|\mathfrak{a}_m^{\nu}|\right)_o;$$

	$(3)$ if there exists a valuation $\tilde\nu$ on $\mathcal{O}_o$ satisfying the above two statements and $\tilde\nu(f)\ge\nu(f)$ for any $(f,o)\in\mathcal{O}_o$, then $\tilde\nu=\nu$.
\end{Theorem}

	If a valuation $\nu$ on $\mathcal{O}_o$  is a Zhou valuation related to $|\mathfrak{q}_0|^2$, its corresponding Zhou weight is given by (\cite{BGMY-valuation}, see also Lemma \ref{cor-approximation})
	\[\Phi_{o,\max}(w)=\sup\left\{\frac{\log|f(w)|}{\nu(f)} : f\in \mathcal{O}(D), \ \sup_D|f|\le 1, \ f(o)=0, \  f\not\equiv 0\right\},\]
	where $D$ is the unit ball in $\mathbb{C}^n$.

The following proposition shows the existence of Zhou valuations with some specific lower bounds.

\begin{Proposition}\label{c:exist}
Let $\nu$ be a  valuation on $\mathcal{O}_o$. Assume that $\nu$ satisfies the statements $(1)$ and $(2)$ in Theorem \ref{the:char}. Then there exists a Zhou valuation $\nu^{\mathrm{Zh}}$ related to $|\mathfrak{q}_0|^2$ such that $\nu^{\mathrm{Zh}}\ge\nu$.
\end{Proposition}

\subsection{Denseness of the Zhou valuations}

A valuation $\nu$ on $\mathcal{O}_o$ is called a \emph{centered valuation}, if $\nu(\mathfrak{m})>0$, where $\mathfrak{m}$ is the maximal ideal of $\mathcal{O}_o$. Denote the set of all centered valuations on $\mathcal{O}_o$ by $\mathrm{Val}_o$, which is a cone (i.e. for any $\lambda>0$ and $\nu\in\mathrm{Val}_o$, it holds that $\lambda \nu\in\mathrm{Val}_o$). There is a natural topology on $\mathrm{Val}_o$, which is the weakest topology such that the functional 
\begin{flalign*}
	\begin{split}
		f\colon\mathrm{Val}_o &\rightarrow \mathbb{R}_{\ge 0}\\
		\nu&\mapsto \nu(f)
	\end{split}
\end{flalign*}
is continuous for every $f\in\mathcal{O}_o^*$. This topology is equivalent to the weakest topology such that the functional
\begin{flalign*}
	\begin{split}
		\mathfrak{a}\colon\mathrm{Val}_o &\rightarrow \mathbb{R}_{\ge 0}\\
		\nu&\mapsto \nu(\mathfrak{a})
	\end{split}
\end{flalign*}
is continuous for every nontrivial ideal $\mathfrak{a}$ of $\mathcal{O}_o$ (see \cite{JM12}). One can see that
\[\mathcal{B}:=\left\{\bigcap_{i=1}^r V(\nu_i,\mathfrak{a}_i,n)\colon r,n\in\mathbb{Z}_{>0}\right\}\]
is a topological basis of $\mathrm{Val}_o$, where $\nu_i\in\mathrm{Val}_o$, $\mathfrak{a}_i$ is an ideal of $\mathcal{O}_0$, and
\[V(\nu_i,\mathfrak{a}_i,n):=\big\{\nu\in\mathrm{Val}_o\colon |\nu(\mathfrak{a}_i)-\nu_i(\mathfrak{a}_i)|<1/n\big\}\]
for each $i$.

Let $\nu\in\mathrm{Val}_o$, and we still adopt the notation
\[\mathscr{A}(\nu):=\sup\big(\nu(\mathfrak{a}_{\bullet})\mathrm{lct}^{\mathfrak{q}}(\mathfrak{a}_{\bullet})-\nu(\mathfrak{q})\big),\]
where the supremum is over all nonzero ideals $\mathfrak{q}$ and graded sequence $\mathfrak{a}_{\bullet}$ satisfying $\mathrm{lct}^{\mathfrak{q}}(\mathfrak{a}_{\bullet})<+\infty$. It still holds that $\mathscr{A}(\nu)\le A(\nu)$, where $A(\nu)$ is the log-discrepancy of the valuation $\nu$.

Denote by $\mathrm{ZVal}_o$ the set of all Zhou valuations on $\mathcal{O}_o$, which is a subcone of $\mathrm{Val}_o$ (see \cite{BGMY-valuation}).

\begin{Theorem}\label{thm.nu.inf.Zhou.val}
    Let $\nu\in\mathrm{Val}_o$. If $\mathscr{A}(\nu)<+\infty$, then
    \begin{equation}\label{eq-nu.q.inf.Zhou.val.q}
        \nu(\mathfrak{q})=\inf\big\{\hat{\nu}(\mathfrak{q}) : \hat{\nu}\ge\nu, \ \hat{\nu}\in \mathrm{ZVal}_o\big\},
    \end{equation}
    for any nonzero ideal $\mathfrak{q}$ of $\mathcal{O}_o$. 
\end{Theorem}

Set
\[\mathrm{Val}_o^{<\infty}:=\mathrm{Val}_o\cap\{\nu : \mathscr{A}(\nu)<+\infty\}.\]
Theorem \ref{thm.nu.inf.Zhou.val} gives a denseness property of Zhou valuations.
\begin{Theorem}\label{thm-ZVal.dense}
    $\mathrm{ZVal}_o$ is dense in $\mathrm{Val}_o^{<\infty}$.
\end{Theorem}

Note that $\mathscr{A}(\nu)= A(\nu)<+\infty$ holds for a quasimonomial valuation $\nu$, and the space of quasimonomial valuations is dense in $\mathrm{Val}_o$ (see \cite{JM12}). Consequently, we can deduce that

\begin{Corollary}
	$\mathrm{ZVal}_o$ is dense in $\mathrm{Val}_o$.
\end{Corollary}

\section{Preparations}

In this section, we do some preparations for the proofs of the theorems.

\subsection{Some results about plurisubharmonic functions}

In this section, we recall some results about plurisubharmonic functions, which will be used in future discussions.

Let $f$ be a holomorphic function near $o$, and let $\phi$  be a plurisubharmonic function near $o$. 
Recall that (see \cite{GZopen-effect,guan-effect}) 
\[C_{f,\phi}(U):=\inf\left\{\int_{U} |\tilde{f}|^{2} :(\tilde{f}-f,o)\in \mathcal{I}(\phi)_{o} \ \& \ \tilde{f}\in\mathcal{O}(U)\right\},\]
where $U\subseteq \Delta^{n}$ is a domain with $o\in U$.

Let $\{\phi_{m}\}_{m\in\mathbb{N}^{+}}$ be a sequence of negative plurisubharmonic functions on $\Delta^{n}$,
which is convergent to a negative Lebesgue measurable function $\phi$ on $\Delta^{n}$ in Lebesgue measure.

In \cite{GZopen-effect}, Guan-Zhou presented the following lower semicontinuity property of
plurisubharmonic functions with a multiplier.

\begin{Proposition}[\cite{GZopen-effect}]
	\label{p:effect_GZ}
	Let $f$ be a holomorphic function near $o$, and let $\varphi_0$ be a plurisubharmonic function near $o$.
	Assume that for any small enough neighborhood $U$ of $o$,
	the pairs $(f,\phi_{m})$ $(m\in\mathbb{N}^+)$ satisfies
	\begin{equation}
		\label{equ:A}
		\inf_{m}C_{f,\varphi_0+\phi_{m}}(U)>0.
	\end{equation}
	Then $|f|^{2}e^{-2\varphi_{0}}e^{-2\phi}$ is not integrable near $o$.
\end{Proposition}

The Noetherian property of multiplier ideal sheaves (see \cite{demailly-book}) shows that
\begin{Remark}
	\label{rem:effect_GZ}
	Assume that
	
	$(1)$ $\phi_{m+1}\geq\phi_{m}$ holds for any $m$;
	
	$(2)$ $|f|^{2}e^{-2\varphi_{0}}e^{-2\phi_{m}}$ is not integrable near $o$ for any $m$.
	
	Then inequality \eqref{equ:A} holds.
\end{Remark}

Proposition \ref{p:effect_GZ} and Remark \ref{rem:effect_GZ} imply the strong openness property $\mathcal{I}(\varphi)=\mathcal{I}_+(\varphi)$ of plurisubharmonic functions $\varphi$.

Recall that for any plurisubharmonic function $\varphi$ near $o$,  the Lelong number $\nu(\varphi,o):=\sup\{c\ge0:\varphi\le c\log|z|$ near $o\}$ (see \cite{demailly-book}). We recall the following famous result due to Skoda.

\begin{Lemma}[\cite{skoda1972}]
	\label{l:skoda72}
	For any plurisubharmonic function $\varphi$ near $o$, if the Lelong number $\nu(\varphi,o)<1$, then $e^{-2\varphi}$ is integrable near $o$.
\end{Lemma}

The following lemma gives a convergence property of jumping numbers.

\begin{Lemma}
	[see \cite{BGMY-valuation}]
	\label{l:lct-N}
	Let $\varphi$ be any plurisubharmonic function near $o$, and let $f\not\equiv0$ be any holomorphic function near $o$. Then
	\[\lim_{N\rightarrow+\infty}c_o^{f}\big(\max\{\varphi,N\log|z|\}\big)=c_o^{f}(\varphi).\]
\end{Lemma}

Let us recall two results about upper envelope of plurisubharmonic functions.

\begin{Lemma}[\emph{Choquet's lemma}, see \cite{demailly-book}]
	\label{lem:Choquet} Every family $(u_{\alpha})$ of uppersemicontinuous functions has a countable subfamily $(v_{j})=(u_{\alpha(j)})$,
	such that its upper envelope $v=\sup_{j}v_{j}$ satisfies $v\leq u\leq u^{*} = v^{*}$,
	where $u=\sup_{\alpha}u_{\alpha}$, $u^{*}(z):=\lim_{\varepsilon\to0}\sup_{\mathbb{B}^{n}(z,\varepsilon)}u$
	and $v^{*}(z):=\lim_{\varepsilon\to0}\sup_{\mathbb{B}^{n}(z,\varepsilon)}v$ are the regularizations of $u$ and $v$.
\end{Lemma}

\begin{Proposition}[see Proposition (4.24) in \cite{demailly-book}]
	\label{pro:Demailly}
	If all $(u_{\alpha})$ are subharmonic, the upper regularization $u^{*}$ is subharmonic
	and equals almost everywhere to $u$.
\end{Proposition}

Let $\varphi$ be a plurisubharmonic function near $o$. Let us recall the subadditive theorem of the multiplier ideal sheaves.
\begin{Lemma}[see \cite{DEL00}]
	\label{l:subaddtive}
	For any positive numbers $t$ and $s$, we have $\mathcal{I}((s+t)\varphi)_o\subset\mathcal{I}(t\varphi)_o\cdot\mathcal{I}(s\varphi)_o$.
\end{Lemma}

The following is the well-known Demailly's approximation theorem of plurisubharmonic functions.

\begin{Lemma}[see \cite{demailly2010}]\label{l:appro-Berg}
	Let $D\subset\mathbb{C}^n$ be a bounded pseudoconvex domain, and let $\varphi$ be a plurisubharmonic function $D$. For any positive integer $m$, let $\{\sigma_{m,k}\}_{k=1}^{\infty}$ be an orthonormal basis of $A^2(D,2m\varphi):=\big\{f\in\mathcal{O}(D):\int_{D}|f|^2e^{-2m\varphi}d\lambda<+\infty\big\}$. Denote that 
	$$\varphi_m:=\frac{1}{2m}\log\sum_{k=1}^{\infty}|\sigma_{m,k}|^2$$
	on $D$. Then there exist two positive constants $c_1$ (depending only on $n$ and diameter of $D$) and $c_2$ such that 
	\begin{equation}
		\label{eq:0911b}\varphi(z)-\frac{c_1}{m}\le\varphi_m(z)\le\sup_{|\tilde z-z|<r}\varphi(\tilde z)+\frac{1}{m}\log\frac{c_2}{r^n}
	\end{equation}
	for any $z\in D$ satisfying $\{\tilde z\in\mathbb{C}^n:|\tilde z-z|<r\}\subset\subset D$. Especially, $\varphi_m$ converges to $\varphi$ pointwisely and in $L^1_{\mathrm{loc}}$ on $D$.
\end{Lemma}

\subsection{Zhou weights}

In this section, we recall some results about Zhou weights (see \cite{BGMY-valuation}).

Let $f_{0}=(f_{0,1},\cdots,f_{0,m})$ be a vector,
where $f_{0,1},\cdots,f_{0,m}$ are holomorphic functions near $o$, and $\varphi_{0}$ a plurisubharmonic function near $o$ such that $|f_{0}|^{2}e^{-2\varphi_{0}}$ is integrable near $o$. Let $D$ be a bounded domain in $\mathbb{C}^n$. Recall the definition of the global Zhou weights.
\begin{Definition}[\cite{BGMY-valuation}]
	Call a negative plurisubharmonic function $\Phi^{f_0,\varphi_0,D}_{o,\max}$ ($\Phi^{D}_{o,\max}$ for short) on $D$ a global Zhou weight related to $|f_0|^2e^{-2\varphi_0}$ if the following statements hold:
	
	$(1)$ $|f_0|^2e^{-2\varphi_0}|z|^{2N_0}e^{-2\Phi^{D}_{o,\max}}$ is integrable near $o$ for large enough $N_0$;
	
	$(2)$ $|f_0|^2e^{-2\varphi_0-2\Phi^{D}_{o,\max}}$ is not integrable near $o$;
	
	$(3)$ for any negative plurisubharmonic function $\tilde\varphi$ on $D$ satisfying that $\tilde\varphi\ge\Phi^{D}_{o,\max}$ on $D$ and $|f_0|^2e^{-2\varphi_0-2\tilde\varphi}$ is not integrable near $o$, $\tilde\varphi=\Phi^{D}_{o,\max}$ holds on $D$.
\end{Definition}

As $D$ is a bounded domain, any global Zhou weight is also a local Zhou weight, and for any local Zhou weight $\Phi_{o,\max}$ related to $|f_0|^2e^{-2\varphi_0}$ there exists a global Zhou weight $\Phi^D_{o,\max}$ related to $|f_0|^2e^{-2\varphi_0}$ such that $\Phi^D_{o,\max}=\Phi_{o,\max}+O(1)$ near $o$ (see \cite{BGMY-valuation}).

\begin{Lemma}[\cite{BGMY-valuation}]\label{l:local-glabal}
	Let $\Phi_{o,\max}$ be a local Zhou weight related to $|f_0|^2e^{-2\varphi_0}$ near $o$, and denote
	\begin{equation}
		\nonumber
		\begin{split}
			L(\Phi_{o,\max}):=\big\{\tilde\varphi(z)\in \mathrm{PSH}(D):\tilde\varphi<0\ \& \ |f_0|^2e^{-2\varphi_0-2\tilde\varphi}\text{ is not integrable near} \ o& \\
			\& \ \tilde\varphi\ge\Phi_{o,\max}+O(1)\text{ near $o$}&\big\}.
		\end{split}
	\end{equation}
	Then 
	\begin{equation}
		\nonumber
		\Phi^D_{o,\max}(z):=\sup\big\{\tilde\varphi(z):\tilde\varphi\in L(\Phi_{o,\max})\big\}, \ \forall z\in D
	\end{equation}
	is a global Zhou weight related to $|f_0|^2e^{-2\varphi_0}$ on $D$ satisfying that 
	$$\Phi^D_{o,\max}=\Phi_{o,\max}+O(1)$$ near $o$.
\end{Lemma}

The following lemma shows the existence of Zhou weight.

\begin{Lemma}[\cite{BGMY-valuation}]\label{r:exists}
	Assume that there exists a negative plurisubharmonic function $\varphi$ on $D$ such that $|f_0|^2e^{-2\varphi_0-2\varphi}|z|^{2N_0}$ is integrable near $o$ for large enough $N_0$ and $(f_0,o)\not\in\mathcal{I}(\varphi+\varphi_0)_o.$
	
	Then there exists a global Zhou weight $\Phi^{D}_{o,\max}$ on $D$ related to $|f_0|^2e^{-2\varphi_0}$ such that $\Phi^{D}_{o,\max}\ge\varphi$ on $D$. 
\end{Lemma}

We recall an approximation result of global Zhou weights.

\begin{Lemma}[\cite{BGMY-valuation}]\label{cor-approximation}
	If $D$ is a bounded strictly hyperconvex domain, and $\Phi^D_{o,\max}$ is a global Zhou weight related to some $|f_0|^2e^{-2\varphi_0}$   on $D$ near $o$, then for any $w\in D$, we have
	\[\Phi^D_{o,\max}(w)=\sup\left\{\frac{\log|f(w)|}{\sigma(\log|f|, \Phi^D_{o,\max})} : f\in \mathcal{O}(D), \ \sup_D|f|\le 1, \ f(o)=0, \  f\not\equiv 0\right\}.\]
\end{Lemma}

\subsection{Some results about valuations}
In this section, we give some lemmas about valuations on $\mathcal{O}_o$.

Let $\mathfrak{q}_0$ be a nonzero ideal of $\mathcal{O}_o$, and let $\nu$ be a  valuation on $\mathcal{O}_o$. Assume that $\nu$ satisfies the statements $(1)$ and $(2)$ in Theorem \ref{the:char}, i.e., $\{(f,o)\in\mathcal{O}_o:v(f)>0\}$ is the maximal ideal of $\mathcal{O}_o$ and  
 $$\mathfrak{q}_0\not\subseteq\mathcal{I}\left(\frac{1}{m}\log|\mathfrak{a}_m^{\nu}|\right)_o$$ for any positive integer $m$,  where $$\mathfrak{a}_m^{\nu}:=\{(f,o)\in\mathcal{O}_o:\nu(f)\ge m\}$$ is an ideal of $\mathcal{O}_o$. Then we can use $\nu$ to construct a plurisubharmonic function $\varphi_{\nu}$.

\begin{Lemma}
	\label{l:valuation--psh} There exists a negative plurisubharmonic function $\varphi_{\nu}$ on $D$ such that for any $f\in\mathcal{O}(D)$ satisfying $|f|\le 1$ on $D$,
	$$\log|f|\le \nu(f)\varphi_{\nu}$$ on $D$,
	$|\mathfrak{q}_0|^2e^{-2\varphi_{\nu}}$ is not integrable near $o$ and $|\mathfrak{q}_0|^2e^{-2\varphi_{\nu}}|z|^N$ is integrable near $o$ for large enough $N$, where $D$ is  the unit ball in $\mathbb{C}^n$.
\end{Lemma}
\begin{proof}
		Denote 
	\begin{equation}
		\nonumber\begin{split}
			\varphi(w)=\sup\Bigg\{\frac{\log|f(w)|}{a_f} : f\in \mathcal{O}(D),& \ \sup_D|f|\le 1, \ f(o)=0, \  f\not\equiv 0\\
			&\&\,a_f\in\mathbb{Q}\cap(0,\nu(f)]\text{ is a constant}\Bigg\}
		\end{split}
	\end{equation}
on $D$.
	 Let $\varphi_{\nu}=\varphi^*$ be the upper semicontinuous regularization of $\varphi$, which is a plurisubharmonic function on $D$. 
	
	By Lemma \ref{lem:Choquet}, there is a sequence  $\left\{\frac{\log|f_j(w)|}{a_{f_j}}\right\}_{j\in\mathbb{Z}_{>0}}$ such that 
	$$\varphi_{\nu}(w)=\left(\sup_j\frac{\log|f_j(w)|}{a_{f_j}}\right)^*$$
	on $D$.
	Denote
	$$\tilde\varphi_k(w):=\max_{1\le j\le k}\frac{\log|f_j(w)|}{a_{f_j}}=\log\max_{1\le j\le k}|f_j(w)|^{\frac{1}{a_{f_j}}}$$ for any positive integer $k$. There exists a positive number $m_k$ such that $\frac{m_k}{a_{f_j}}\in\mathbb{Z}_{>0}$ for any $1\le j\le k$. Thus, $$\tilde\varphi_k(w)=\frac{1}{m_k}\log\max_{1\le j\le k}\big|f_j(w)^{\frac{m_k}{a_{f_j}}}\big|$$ and $$\nu\big(f_j^{\frac{m_k}{a_{f_j}}}\big)={\frac{m_k}{a_{f_j}}}\nu(f_j)\ge m_k.$$
	Then Statement $(2)$ shows that $|\mathfrak{q}_0|^2e^{-2\tilde\varphi_k}$ is not integrable near $o$ for any $k$. Note that $\tilde\varphi_k$ is increasingly convergent to $\varphi_{\nu}$ a.e. by Proposition \ref{pro:Demailly}. It follows from Proposition \ref{p:effect_GZ} and Remark \ref{rem:effect_GZ} that 
	$|\mathfrak{q}_0|^2e^{-2\varphi_{\nu}}$
	is not integrable near $o$.
	By definition, we have 
	\begin{equation}
		\nonumber
		\log|f|\le \nu(f)\varphi_{\nu}
	\end{equation}
on $D$,
which shows that 
$$\max_{1\le j\le n}\frac{1}{\nu(z_j)}\log|z_j|\le \varphi_{\nu},$$
where $(z_1,\ldots,z_m)$ is the coordinate system near $o$ and $\nu(z_j)>0$ for any $1\le j\le n$. Thus,  $|\mathfrak{q}_0|^2e^{-2\varphi_{\nu}}|z|^N$ is integrable near $o$ for large enough $N$.
\end{proof}

We recall
a closedness property of the ideals of $\mathcal O_{o}$.

\begin{Lemma}[see \cite{G-R}]
	\label{closedness}
	Let $\mathfrak{a}$ be an ideal of $\mathcal O_{o}$, and let $\{f_j\}_j$ be a sequence of  holomorphic functions on an open neighborhood $U$ of  $o$. Assume that all germs $(f_{j},o)\in\mathfrak{a}$ and the $f_j$ converge uniformly on $U$ towards  a holomorphic function $f$. Then $(f,o)\in \mathfrak{a}$.	
\end{Lemma}

The following lemma will be used in the proof of Theorem \ref{the:char}.

\begin{Lemma}
	\label{l:1029_1}
	Let $\nu$ be a valuation on $\mathcal{O}_o$, and let $\{f_j\}_j$ be a sequence of  holomorphic functions on an open neighborhood $U$ of  $o$. Assume that  $f_j$ converge uniformly on $U$ towards  a holomorphic function $f$. Then $\nu(f)\ge\limsup_{j\rightarrow+\infty}\nu(f_j)$.
	
	Especially, assume that $\{(h,o)\in\mathcal{O}_o:\nu(h)>0\}$ is the maximal ideal of $\mathcal{O}_o$, and let $g(w)=\sum_{\alpha\in\mathbb{Z}_{\ge0}^{n}}a_{\alpha}w^{\alpha}$ (Taylor expansion) be a holomorphic function near $o$, then we have 
	$$\nu(g)=\lim_{j\rightarrow+\infty}\nu(g_j),$$
	 where $g_j(w):=\sum_{|\alpha|\le j}a_{\alpha}w^{\alpha}$ and $|\alpha|:=\sum_{1\le k\le n}\alpha_k$ for any $\alpha\in\mathbb{Z}_{\ge0}^{n}.$
\end{Lemma}
\begin{proof}
	For any $$x<\limsup_{j\rightarrow+\infty}\nu(f_j),$$ there exists a subsequence of $\{f_j\}_j$ denoted also by $\{f_j\}_j$ such that $\nu(f_j)>x$ for any $j$. Note that $\{(h,o)\in\mathcal{O}_o:\nu(h)>x\}$ is an ideal of $\mathcal{O}_o$ and $f_j$ converge uniformly  towards   $f$, then it follows from Lemma \ref{l:1029_1} that $\nu(f)>x$. Thus, we have $$\nu(f)\ge\limsup_{j\rightarrow+\infty}\nu(f_j).$$

	Note that $g_j$ converge uniformly  towards   $g$ on a neighborhood of $o$, thus we have 
	\begin{equation}
		\label{eq:1029aa}\nu(g)\ge\limsup_{j\rightarrow+\infty}\nu(g_j).
	\end{equation}
	As $\{(h,o)\in\mathcal{O}_o:\nu(h)>0\}$ is the maximal ideal of $\mathcal{O}_o$, we have $$c:=\min_{1\le l\le n}\{\nu(w_l)\}>0,$$
	which implies that 
	\begin{equation}
		\nonumber\begin{split}
			\nu(g-g_j)&\ge\limsup_{k\rightarrow+\infty}\nu(g_k-g_j)\\
			&\ge\limsup_{k\rightarrow+\infty}\min_{j<|\alpha|\le k}\{\nu(a_{\alpha}w^{\alpha})\}\\
			&\ge jc.
		\end{split}
	\end{equation}
 Take $j_0$ such that $j_0c>\nu(g)$. For any $j>j_0$, we have 
$$\nu(g_j)\ge\min\{\nu(g),\nu(g-g_j)\}=\nu(g).$$
Combining equality \eqref{eq:1029aa}, we have 
$$\nu(g)=\lim_{j\rightarrow+\infty}\nu(g_j).$$

Thus, Lemma \ref{l:1029_1} holds.
\end{proof}

\section{Proofs of Theorem \ref{thm:compute} and Remark \ref{r:A(nu)}}

In this section, we prove Theorem \ref{thm:compute} and Remark \ref{r:A(nu)}.

Firstly,  we recall some results.

Given a local Zhou weight $\Phi_{o,\max}$ near $o$ related to $|f_0|^2$, the following lemma shows the jumping number $c_o^{G}(\Phi_{o,\max})$ and the Zhou valuation $\nu(G,\Phi_{o,\max}):=\sigma(\log|G|,\Phi_{o,\max})$ for all $(G,o)\in\mathcal{O}_o$ are linearly controlled by each other.

\begin{Lemma}[\cite{BGMY-valuation}]
	\label{thm:valu-jump}
	For any holomorphic function $G$ near $o$,	
	we have the following relation between the jumping number $c^G_o(\Phi_{o,\max})$ and the Zhou valuation $\nu(\cdot,\Phi_{o,\max})$,
	\begin{equation*}
		\begin{split}
			\nu(G,\Phi_{o,\max})+c_o(\Phi_{o,\max})
			\le& c^G_o(\Phi_{o,\max})\\
			\le& \nu(G,\Phi_{o,\max})-\sigma(\log|f_0|,\Phi_{o,\max})+1.
		\end{split}
	\end{equation*}
	Especially, if $|f_0|^2\equiv1$, we have 
	$$\nu(G,\Phi_{o,\max})+1=c^G_o(\Phi_{o,\max}).$$
\end{Lemma}

Let $\psi$ be a plurisubharmonic function near $o$. Recall that the Tian function (see \cite{BGMY-valuation}) was defined as 
$$\mathrm{Tn}(t;\psi):=\sup\{c:|\mathfrak{q}_0|^2e^{2t\psi-2c\Phi_{o,\max}}\text{ is integrable near }o\}$$
for any $t\in\mathbb{R}$.
Then we have 
\begin{Lemma}[\cite{BGMY-valuation}]
	\label{l:tianfunction}
	$\mathrm{Tn}(t;\psi)=\mathrm{Tn}(0;\psi)+\sigma(\psi,\Phi_{o,\max})t$
	holds for any $t\ge0$.
\end{Lemma}

Let $\nu$ be a Zhou valuation related to $|\mathfrak{q}_0|^2$, where $\mathfrak{q}_0$ is a nonzero ideal.  Let $\mathfrak{a}^{\nu}_{\bullet}=(\mathfrak{a}^{\nu}_m)_{m\ge1}$ be a graded sequence of ideals such that $\mathfrak{a}_m^{\nu}=\{(f,o)\in\mathcal{O}_o:\nu(f)\ge m\}$.
\begin{Lemma}\label{l:1018a}
	$\nu(\mathfrak{a}^{\nu}_{\bullet})=1$ and $\mathrm{lct}^{\mathfrak{q}}(\mathfrak{a}^{\nu}_{\bullet})=c_o^{\mathfrak{q}}(\Phi_{o,\max})$, where $\mathfrak{q}$ is any nonzero ideal of $\mathcal{O}_o$ and $\Phi_{o,\max}$ is the corresponding local Zhou weight of $\nu$.
\end{Lemma}
\begin{proof}
	Note that $\nu(\mathfrak{a}^{\nu}_m)\ge m$, which deduces $\nu(\mathfrak{a}^{\nu}_{\bullet})\ge1$. Let $(f_1,o)\in\mathcal{O}_o$ such that $\nu(f_1)\in(0,+\infty)$. Then $\nu(f_1^k)=k\nu(f_1)$ for any positive integer $k$. As $\mathfrak{a}^{\nu}_m=\{(f,o)\in\mathcal{O}_o:v(f)\ge m\}$, we have $\nu(\mathfrak{a}_m)\le m+\nu(f_1)$ for any $m$. Thus, we get $\nu(\mathfrak{a}^{\nu}_{\bullet})=1.$
	
	Note that $\nu(f)=\sigma(\log|f|,\Phi_{o,\max}):=\sup\{c\ge0:\log|f|\le c\Phi_{o,\max}+O(1)$ near $o\}$.
	It follows from Lemma \ref{thm:valu-jump} that there exists a constant $C$ such that $\mathcal{I}(m\Phi_{o,\max})_o\subset\mathfrak{a}_{m+C}^{\nu}$ for any $m$, which shows 
	$$\Phi_{o,\max}\le \frac{1}{m}\log|\mathcal{I}(m\Phi_{o,\max})_o|+O(1)\le \frac{1}{m}\log|\mathfrak{a}_{m+C}^{\nu}|+O(1)$$
	near $o$ by Lemma \ref{l:appro-Berg}. Then we have 
	$$c_o^{\mathfrak{q}}(\Phi_{o,\max})\le \liminf_{m\rightarrow+\infty}c_o^{\mathfrak{q}}(\frac{1}{m}\log|\mathfrak{a}_{m+C}^{\nu}|)=\mathrm{lct}^{\mathfrak{q}}(\mathfrak{a}^{\nu}_{\bullet}).$$
	On the other hand, note that $\frac{1}{m}\log|\mathfrak{a}_m^{\nu}|\le \Phi_{o,\max}+O(1)$ near $o$, which implies  $\mathrm{lct}^{\mathfrak{q}}(\mathfrak{a}^{\nu}_{\bullet})\le c_o^{\mathfrak{q}}(\Phi_{o,\max}).$
	Thus, we obtain $$\mathrm{lct}^{\mathfrak{q}}(\mathfrak{a}^{\nu}_{\bullet})=c_o^{\mathfrak{q}}(\Phi_{o,\max}).$$
\end{proof}

Now, we prove Theorem \ref{thm:compute}.

\begin{proof}[Proof of Theorem \ref{thm:compute}]
	By definition of $\mathscr{A}(\nu)$, it suffices to prove
	$$\nu(\mathfrak{a}_{\bullet})\mathrm{lct}^{\mathfrak{q}}(\mathfrak{a}_{\bullet})-\nu(\mathfrak{q})\le \nu(\mathfrak{a}^{\nu}_{\bullet})\mathrm{lct}^{\mathfrak{q}_0}(\mathfrak{a}^{\nu}_{\bullet})-\nu(\mathfrak{q}_0)$$
	for any nonzero ideal $\mathfrak{q}$ and graded sequence $\mathfrak{a}_{\bullet}$ such that $\mathrm{lct}^{\mathfrak{q}}(\mathfrak{a})<+\infty$.
	$\mathrm{lct}^{\mathfrak{q}}(\mathfrak{a}_{\bullet})<+\infty$ shows that $\lim_{m\rightarrow+\infty}\mathrm{lct}^{\mathfrak{q}}(\mathfrak{a}_m)=0$. Note that
	\begin{equation}
		\label{eq:1018a}
		\begin{split}
			\nu(\mathfrak{a}_{\bullet})\mathrm{lct}^{\mathfrak{q}}(\mathfrak{a}_{\bullet})&=\lim_{m\rightarrow+\infty}\nu(\mathfrak{a}_{m})\mathrm{lct}^{\mathfrak{q}}(\mathfrak{a}_{m})\\
			&=\lim_{m\rightarrow+\infty}[\nu(\mathfrak{a}_{m})]\mathrm{lct}^{\mathfrak{q}}(\mathfrak{a}_{m})\\
			&\le \lim_{m\rightarrow+\infty}[\nu(\mathfrak{a}_{m})]\mathrm{lct}^{\mathfrak{q}}(\mathfrak{a}^{\nu}_{[\nu(\mathfrak{a}_{m})]}),
		\end{split}
	\end{equation}
	where $[x]:=\sup\{y\in\mathbb{Z}:y\le x\}$ for any $x\in\mathbb{R}$.
	There exists a positive number $N$ such that $\Phi_{o,\max}\ge N\log|z|+O(1)$ near $o$. As $\lim_{m\rightarrow+\infty}\mathrm{lct}^{\mathfrak{q}}(\mathfrak{a}_m)=0$, for any $N_1>0$ there exists $N_2>0$ such that
	\[\log|\mathfrak{a}_m|\le N_1\log|z|+O(1)\le \frac{N_1}{N}\Phi_{o,\max}+O(1)\]
	near $o$ for any $m>N_2$, which implies that $\lim_{m\rightarrow+\infty}\nu(\mathfrak{a}_m)=+\infty$.

	Using inequality \eqref{eq:1018a}, it suffices to consider the case $\mathfrak{a}_{\bullet}=\mathfrak{a}_{\bullet}^{\nu}$.
	Note that $c_o^{\mathfrak{q}_0}(\Phi_{o,\max})=1$. By Lemma \ref{l:1018a}, we only need to prove that the equality $$c_o^{\mathfrak{q}}(\Phi_{o,\max})-\nu(\mathfrak{q})\le 1-\nu(\mathfrak{q}_0)$$
	holds
	for any nonzero ideal $\mathfrak{q}$.
	Assume that $\mathfrak{q}$ is generated by $f_1,\ldots,f_l$. Then Lemma \ref{thm:valu-jump} shows that 
	$$c_o^{\mathfrak{q}}(\Phi_{o,\max})= \min_jc_o^{{f_j}}(\Phi_{o,\max})\le\min_j\{ \nu(f_j)+1-\nu(\mathfrak{q}_0)\}=\nu(\mathfrak{q})+1-\nu(\mathfrak{q}_0).$$
	Thus, Theorem \ref{thm:compute} holds.
\end{proof}

Following the notations in Theorem \ref{thm:compute},  we prove Remark \ref{r:A(nu)}.

\begin{proof}[Proof of Remark \ref{r:A(nu)}]
	Theorem \ref{thm:compute} shows that $$\mathscr{A}(\nu)=\nu(\mathfrak{a}^{\nu}_{\bullet})\mathrm{lct}^{\mathfrak{q}_0}(\mathfrak{a}^{\nu}_{\bullet})-\nu(\mathfrak{q}_0).$$ Lemma \ref{l:1018a} indicates $\nu(\mathfrak{a}^{\nu}_{\bullet})=1$ and $\mathrm{lct}^{\mathfrak{q}_0}(\mathfrak{a}^{\nu}_{\bullet})=c_o^{\mathfrak{q}_0}(\Phi_{o,\max})$, where $\Phi_{o,\max}$ is the corresponding local Zhou weight related to $|\mathfrak{q}_0|^2$ of $\nu$. Thus,
	\[c_o^{\mathfrak{q}_0}(\Phi_{o,\max})=1 \ \&  \ \mathscr{A}(\nu)=1-\nu(\mathfrak{q}_0).\]

	For any $(h,o)\in\mathcal{O}^*_o$, we have 
	$$c_o^{h}(\Phi_{o,\max})-\nu(h)=\nu(\mathfrak{a}^{\nu}_{\bullet})\mathrm{lct}^{(h)_o}(\mathfrak{a}^{\nu}_{\bullet})-\nu((h)_o)\le\mathscr{A}(\nu),$$
	where $(h)_o$ is the ideal of $\mathcal{O}_o$ generated by $(h,o)$.
	There exists $(h_1,o)\in\mathfrak{q}_0$ such that $\nu(h_1)=\nu(\mathfrak{q}_0)$. It follows from Lemma \ref{thm:valu-jump} that 
	$$c_o^{h_1}(\Phi_{o,\max})\le\nu(h_1)-\nu(\mathfrak{q}_0)+1=1.$$
	As $c_o^{\mathfrak{q}_0}(\Phi_{o,\max})=1$, we have $c_o^{h_1}(\Phi_{o,\max})=1$. Thus, $$\mathscr{A}(\nu)=\sup_{(h,o)\in\mathcal{O}^*_o}(c_o^{h}(\Phi_{o,\max})-\nu(h)).$$
	
	Note that $\sigma(\log|\mathfrak{q}_0|,\Phi_{o,\max})=\nu(\mathfrak{q}_0)$ and $c_o^{\mathfrak{q}_0}(\Phi_{o,\max})=1$. Taking $\psi=\log|\mathfrak{q}_0|$, Lemma \ref{l:tianfunction} implies that
	$$\mathrm{Tn}(t)=1+\nu(\mathfrak{q}_0)(t-1)=\nu(\mathfrak{q}_0)t+\mathscr{A}(\nu).$$ 
\end{proof}

\section{Graded sequences, subadditive sequences and Zhou valuations}

In this section, we show the effects of Zhou valuations acting on graded sequences and subadditive sequences, which will be used in the following proofs in the next section.

Let $\mathfrak{a}_{\bullet}$ be a graded sequence of ideals, and let $\mathfrak{b}_{\bullet}$ be the corresponding subadditive sequence of asymptotic multiplier ideals of $\mathfrak{a}_{\bullet}$.
The following lemma gives a relation between $\nu(\mathfrak{a}_{m})$ and $\nu(\mathfrak{b}_m)$ for any Zhou valuation $\nu$.

\begin{Lemma}\label{l:a.b.}
	 For any $m$ such that $\mathfrak{a}_m$ is nonzero, we have 
	$$\nu(\mathfrak{a}_{\bullet})-\frac{\mathscr{A}(\nu)}{m}<\frac{\nu(\mathfrak{b}_m)}{m}\le\frac{\nu(\mathfrak{a}_m)}{m}$$
	for any Zhou valuation $\nu$. In particular, $\nu(\mathfrak{a}_{\bullet})=\nu(\mathfrak{b}_{\bullet})$.
\end{Lemma}

\begin{proof}
	Let $\nu$ be a Zhou valuation related to $|\mathfrak{q}|^2$, whose corresponding local Zhou weight is denoted by $\Phi_{o,\max}$.
	It follows from $\mathfrak{a}_{m}\subset\mathfrak{b}_m$ that $$\nu(\mathfrak{a}_m)\ge\nu(\mathfrak{b}_m)$$
	for any $m$. For any $m$ such that $\mathfrak{a}_m$ is nonzero, there exists $m_1$ large enough such that $\mathfrak{b}_m=\mathcal{I}(\frac{m}{m_1}\log|\mathfrak{a}_{m_1}|)_o$, then it follows from Lemma \ref{thm:valu-jump} and $\mathscr{A}(\nu)=1-\nu(\mathfrak{q})$ that
	\begin{equation}
		\nonumber
		\begin{split}
			\frac{\nu(\mathfrak{b}_m)}{m}&=\frac{\nu(\mathcal{I}(\frac{m}{m_1}\log|\mathfrak{a}_{m_1}|)_o)}{m}\\
			&\ge \frac{c_o^{\mathcal{I}(\frac{m}{m_1}\log|\mathfrak{a}_{m_1}|)_o}(\Phi_{o,\max})-\mathscr{A}(\nu)}{m}\\
			&\ge\frac{c_o^{\mathcal{I}(\frac{m}{m_1}\log|\mathfrak{a}_{m_1}|)_o}(\frac{\log|\mathfrak{a}_{m_1}|}{\nu(\mathfrak{a}_{m_1})})-\mathscr{A}(\nu)}{m}\\
			&> \frac{\frac{m\nu(\mathfrak{a}_{m_1})}{m_1}-\mathscr{A}(\nu)}{m}\\
			&\ge\nu(\mathfrak{a}_{\bullet})-\frac{\mathscr{A}(\nu)}{m},
		\end{split}
	\end{equation}
where the above ``$>$" follows from the strong openness property of the multiplier ideal sheaves(see \cite{GZopen-c}, see also Proposition \ref{p:effect_GZ} and Remark \ref{rem:effect_GZ}).
	Thus, Lemma \ref{l:a.b.} holds.
\end{proof}

Recall that 
\begin{equation}
	\label{eq:1027a}
	\mathscr{A}(\nu):=\sup\{\nu(\mathfrak{a}_{\bullet})\mathrm{lct}^{\mathfrak{q}}(\mathfrak{a}_{\bullet})-\nu(\mathfrak{q}):\text{for any $\mathfrak{q}$ and $\mathfrak{a}_{\bullet}$ such that $\mathrm{lct}^{\mathfrak{q}}(\mathfrak{a}_{\bullet})<+\infty$}\},
\end{equation} 
where $\mathfrak{a}_{\bullet}$ is any graded sequence of ideals. The following remark shows that equality \eqref{eq:1027a} also holds when replacing $\mathfrak{a}_{\bullet}$ by   subadditive sequences of ideals $\mathfrak{b}_{\bullet}$.

\begin{Remark}
	\label{r:sup-b}
	Let $\nu$ be a Zhou valuation related to $|\mathfrak{q}_0|^2$, where $\mathfrak{q}_0$ is a nonzero ideal. Then 
	$$\mathscr{A}(\nu)=\sup\big\{\nu(\mathfrak{b}_{\bullet})\mathrm{lct}^{\mathfrak{q}}(\mathfrak{b}_{\bullet})-\nu(\mathfrak{q}):\text{for any $\mathfrak{q}$ and $\mathfrak{b}_{\bullet}$ such that $\mathrm{lct}^{\mathfrak{q}}(\mathfrak{b}_{\bullet})<+\infty$}\big\}.$$
\end{Remark}
\begin{proof}
	For any nonzero ideal $\mathfrak{q}$ and any subadditive sequence $\mathfrak{b}_{\bullet}$ satisfying $\mathrm{lct}^{\mathfrak{q}}(\mathfrak{b}_{\bullet})<+\infty$, if $\nu(\mathfrak{b}_{\bullet})>0$, it is clear that 
	\begin{equation}
		\nonumber
		\begin{split}
			\nu(\mathfrak{b}_{\bullet})\mathrm{lct}^{\mathfrak{q}}(\mathfrak{b}_{\bullet})=&\lim_{t\rightarrow+\infty}\nu(\mathfrak{b}_t)\mathrm{lct}^{\mathfrak{q}}(\mathfrak{b}_t)\\
			\le&\lim_{t\rightarrow+\infty}[\nu(\mathfrak{b}_t)]\mathrm{lct}^{\mathfrak{q}}(\mathfrak{a}^{\nu}_{[\nu(\mathfrak{b}_t)]})\\
			=&\nu(\mathfrak{a}^{\nu})\mathrm{lct}^{\mathfrak{q}}(\mathfrak{a}^{\nu}_{\bullet}),
		\end{split}
	\end{equation}
	where $[s]:=\sup\{x\in\mathbb{Z}:x\le s\}$ for any $s\in\mathbb{R}$. Thus, by the definition of $\mathscr{A}(\nu)$ we have $$\nu(\mathfrak{b}_{\bullet})\mathrm{lct}^{\mathfrak{q}}(\mathfrak{b}_{\bullet})-\nu(\mathfrak{q})\le \mathscr{A}(\nu)$$ 
	for any $\mathfrak{q}$ and $\mathfrak{b}_{\bullet}$ such that $\mathrm{lct}^{\mathfrak{q}}(\mathfrak{b}_{\bullet})<+\infty$. Note that for any graded sequence of ideals $\mathfrak{a}_{\bullet}$, the corresponding subadditive sequence $\mathfrak{b}_{\bullet}$ of asymptotic multiplier ideals of $\mathfrak{a}_{\bullet}$ satisfies that $\nu(\mathfrak{a}_{\bullet})=\nu(\mathfrak{b}_{\bullet})$ and $\mathrm{lct}^{\mathfrak{q}}(\mathfrak{a}_{\bullet})=\mathrm{lct}^{\mathfrak{q}}(\mathfrak{b}_{\bullet})$. Thus, using the definition of $\mathscr{A}(\nu)$, we get
	$$\mathscr{A}(\nu)=\sup\{\nu(\mathfrak{b}_{\bullet})\mathrm{lct}^{\mathfrak{q}}(\mathfrak{b}_{\bullet})-\nu(\mathfrak{q}):\text{for any $\mathfrak{q}$ and $\mathfrak{b}_{\bullet}$ such that $\mathrm{lct}^{\mathfrak{q}}(\mathfrak{b}_{\bullet})<+\infty$}\}.$$
\end{proof}

Let $\varphi$ be a plurisubharmonic function near $o$. Denote that $$\mathfrak{b}_t^{\varphi}:=\mathcal{I}(t\varphi)_o$$
for any $t>0$.
Lemma \ref{l:subaddtive} (the subadditive theorem, see \cite{DEL00}) shows that 
 $\mathfrak{b}^{\varphi}_{\bullet}=(\mathfrak{b}^{\varphi}_t)_{t>0}$ is a subadditive system of ideals satisfying that $\mathfrak{b}^{\varphi}_{s+t}\subset\mathfrak{b}^{\varphi}_s\cdot\mathfrak{b}^{\varphi}_t$ for any $s,t\in\mathbb{R}_{>0}$. Then we have

\begin{Lemma}
	\label{l:contro}
	Let $\nu$ be a  Zhou valuation, whose corresponding local Zhou weight is denoted by $\Phi_{o,\max}$. Then
	$$\nu(\mathfrak{b}_{\bullet}^{\varphi})=\sigma(\varphi,\Phi_{o,\max})$$ and 
	$$\frac{\nu(\mathfrak{b}^{\varphi}_t)}{t}>\nu(\mathfrak{b}^{\varphi}_{\bullet})-\frac{\mathscr{A}(\nu)}{t}.$$
\end{Lemma}
\begin{proof}
	It follows from Lemma \ref{l:appro-Berg} that $\varphi\le \frac{1}{t}\log|\mathcal{I}(t\varphi)_o|+O(1)$ near $o$ for any $t>0$, which deduces that 
	\begin{equation}
		\nonumber\begin{split}
			\sigma(\varphi,\Phi_{o,\max})&\ge\limsup_{t\rightarrow+\infty}\sigma(\frac{1}{t}\log|\mathcal{I}(t\varphi)_o|,\Phi_{o,\max})\\
			&=\limsup_{t\rightarrow+\infty}\frac{\nu(\mathfrak{b}^{\varphi}_t)}{t}
			\\&=\nu(\mathfrak{b}_{\bullet}^{\varphi}).
		\end{split}
	\end{equation}
	Note that $\varphi\le \sigma(\varphi,\Phi_{o,\max})\Phi_{o,\max}+O(1)$ near $o$, then it follows from Lemma \ref{thm:valu-jump} that 
	\begin{equation}
		\nonumber\begin{split}
			\nu(\mathfrak{b}_{\bullet}^{\varphi})&=\lim_{t\rightarrow+\infty}\frac{\nu(\mathfrak{b}^{\varphi}_t)}{t}\\
			&=\lim_{t\rightarrow+\infty}\frac{c_o^{\mathcal{I}(t\varphi)_o}(\Phi_{o,\max})}{t}\\
			&\ge\lim_{t\rightarrow+\infty}\frac{c_o^{\mathcal{I}(t\varphi)_o}(\frac{\varphi}{\sigma(\varphi,\Phi_{o,\max})})}{t}\\
			&\ge\sigma(\varphi,\Phi_{o,\max}).
		\end{split}
	\end{equation}
	Thus, we have $\nu(\mathfrak{b}_{\bullet}^{\varphi})=\sigma(\varphi,\Phi_{o,\max}).$ Following from Proposition \ref{p:effect_GZ} and Remark \ref{rem:effect_GZ}, $$c_o^{\mathfrak{b}^{\varphi}_t}(\varphi)>t$$ for any $t>0$. Using Lemma \ref{thm:valu-jump}, we have 
	\begin{equation}
		\nonumber
		\begin{split}
			\frac{\nu(\mathfrak{b}^{\varphi}_t)}{t}&\ge \frac{c_o^{\mathfrak{b}^{\varphi}_t}(\Phi_{o,\max})-\mathscr{A}(\nu)}{t}\\
			&\ge\frac{c_o^{\mathfrak{b}^{\varphi}_t}(\frac{\varphi}{\sigma(\varphi,\Phi_{o,\max})})-\mathscr{A}(\nu)}{t}\\
			&>\sigma(\varphi,\Phi_{o,\max})-\frac{\mathscr{A}(\nu)}{t}\\
			&=\nu(\mathfrak{b}^{\varphi}_{\bullet})-\frac{\mathscr{A}(\nu)}{t}.
		\end{split}
	\end{equation}
	Thus, Lemma \ref{l:contro} holds.
\end{proof}

Remark \ref{r:sup-b} shows that 		$$\mathrm{lct}^{\mathfrak{q}}(\mathfrak{b}_{\bullet})\le  \inf_{\nu}\frac{\mathscr{A}(\nu)+\nu(\mathfrak{q})}{\nu(\mathfrak{b}_{\bullet})},$$ where the infimum is over all Zhou valuations $\nu$. The following Proposition tells us that it is actually an equality.

\begin{Proposition}\label{p:inf-b}
	Let $\mathfrak{b}_{\bullet}$ be a subadditive system of ideals,  and let $\mathfrak{q}$ be a nonzero ideal. Then
	\begin{equation}
		\label{eq:1023a}
		\mathrm{lct}^{\mathfrak{q}}(\mathfrak{b}_{\bullet})= \inf_{\nu}\frac{\mathscr{A}(\nu)+\nu(\mathfrak{q})}{\nu(\mathfrak{b}_{\bullet})},
	\end{equation}
	where the infimum is over all Zhou valuations $\nu$. 
\end{Proposition}

\begin{proof}
	Firstly, we prove that $\mathrm{lct}^{\mathfrak{q}}(\mathfrak{b}_{\bullet})=+\infty$ holds if and only if $\nu(\mathfrak{b}_{\bullet})=0$ for any Zhou valuation $\nu$. If $\nu(\mathfrak{b}_{\bullet})>0$ for some Zhou valuation $\nu$ (its corresponding local Zhou weight is denoted by $\Phi_{o,\max}$), there exists $c>0$ such that $\nu(\mathfrak{b}_t)>tc$ for large enough $t$, i.e., $\log|\mathfrak{b}_t|<tc\Phi_{o,\max}+O(1)$ near $o$, which implies that 
	\begin{equation}
		\nonumber\begin{split}
			\mathrm{lct}^{\mathfrak{q}}(\mathfrak{b}_{\bullet})&=\lim_{t\rightarrow+\infty}t\mathrm{lct}^{\mathfrak{q}}(\mathfrak{b}_{t})\\
			&\le\lim_{t\rightarrow+\infty}tc_o^{\mathfrak{q}}(tc\Phi_{o,\max})\\
			&=\frac{c_o^{\mathfrak{q}}(\Phi_{o,\max})}{c}\\
			&<+\infty.
		\end{split}
	\end{equation}
	Thus, $\mathrm{lct}^{\mathfrak{q}}(\mathfrak{b}_{\bullet})=+\infty$ implies $\nu(\mathfrak{b}_{\bullet})=0$ for any Zhou valuation $\nu$. Note that the Lelong number $\nu_L:=\nu(\log|\cdot|,o)$ is also a Zhou valuation.  
	If $\nu_L(\mathfrak{b}_{\bullet})=0$, i.e., $$\lim_{t\rightarrow+\infty}\nu(\frac{1}{t}\log|\mathfrak{b}_t|,o)=0,$$
	it follows from Lemma \ref{l:skoda72} that 
	$$\mathrm{lct}^{\mathfrak{q}}(\mathfrak{b}_{\bullet})=\lim_{t\rightarrow+\infty}c_o^{\mathfrak{q}}(\frac{1}{t}\log|\mathfrak{b}_t|)=+\infty.$$ 
	
	In the following, we prove equality \eqref{eq:1023a} with assumption $\mathrm{lct}^{\mathfrak{q}}(\mathfrak{b}_{\bullet})<+\infty$. Remark \ref{r:sup-b} shows that
	\begin{equation}
		\label{eq:1023d}\mathrm{lct}^{\mathfrak{q}}(\mathfrak{b}_{\bullet})\le \frac{\mathscr{A}(\nu)+\nu(\mathfrak{q})}{\nu(\mathfrak{b}_{\bullet})}
	\end{equation} 
	for any Zhou valuation $\nu$. 
	
	Let $x>\mathrm{lct}^{\mathfrak{q}}(\mathfrak{b}_{\bullet})$ be a constant. Following from the definition of $\mathrm{lct}^{\mathfrak{q}}(\mathfrak{b}_{\bullet})$ and Lemma \ref{l:lct-N}, there exist two constants $t$ and $N$ large enough satisfying
	\begin{equation}
		\nonumber
		tc_o^{\mathfrak{q}}(\max\{\log|\mathfrak{b}_t|,N\log|z|\})<x.
	\end{equation} 
	For simplicity, denote $\max\{\log|\mathfrak{b}_t|,N\log|z|\}$ by $\varphi_{t,N}$.
	The strong openness property of multiplier ideal sheaves (see Proposition \ref{p:effect_GZ} and Remark \ref{rem:effect_GZ}) shows that $$|\mathfrak{q}|^2e^{-2c_o^{\mathfrak{q}}(\varphi_{t,N})\varphi_{t,N}}$$ is not integrable near $o.$ By remark \ref{r:exists}, there exists a local Zhou weight $\Phi_{o,\max}$ (whose corresponding Zhou valuation is denoted by $\nu_{t,N}$) related to $|\mathfrak{q}|^2$ satisfying that $$\Phi_{o,\max}\ge c_o^{\mathfrak{q}}(\varphi_{t,N})\varphi_{t,N}$$ 
	near $o$ and $|\mathfrak{q}|^2e^{-2\Phi_{o,\max}}$ is not integrable near $o.$ It is clear that $\sigma(\varphi_{t,N},\Phi_{o,\max}):=\sup\{c\ge0:\varphi_{t,N}\le c\Phi_{o,\max}+O(1)$ near $o\}=\frac{1}{c_o^{\mathfrak{q}}(\varphi_{t,N})}$.
	Note that $\mathscr{A}(\nu_{t,N})+\nu_{t,N}(\mathfrak{q})=1$.
	Thus, we have
	$$x> tc_o^{\mathfrak{q}}(\varphi_{t,N})=\frac{t}{\sigma(\varphi_{t,N},\Phi_{o,\max})}>\frac{\mathscr{A}(\nu_{t,N})+\nu_{t,N}(\mathfrak{q})}{\frac{\sigma(\log|\mathfrak{b}_t|,\Phi_{o,\max})}{t}}\ge\frac{\mathscr{A}(\nu_{t,N})+\nu_{t,N}(\mathfrak{q})}{\nu_{t,N}(\mathfrak{b}_{\bullet})}.$$
	Combining with inequality \eqref{eq:1023d}, we have
	$$\mathrm{lct}^{\mathfrak{q}}(\mathfrak{b}_{\bullet})=\inf_{\nu} \frac{\mathscr{A}(\nu)+\nu(\mathfrak{q})}{\nu(\mathfrak{b}_{\bullet})},$$
	where the infimum is over all Zhou valuations $\nu$. 
\end{proof}

Using Proposition \ref{p:inf-b}, we have the following remark. 

\begin{Remark}
	Note that for any local Zhou weight $\Phi_{o,\max}$, there exists $N\gg0$ such that $N\log|z|+O(1)\le \Phi_{o,\max}\le\frac{1}{N}\log|z|+O(1)$ near $o$. Thus Proposition \ref{p:inf-b} implies the following three statements are equivalent:
	
	$(1)$ $\mathrm{lct}^{\mathfrak{q}}(\mathfrak{b}_{\bullet})=+\infty$;
	
	$(2)$ $\nu(\mathfrak{b}_{\bullet})=0$ for any Zhou valuation $\nu$;
	
	$(3)$ there exists a Zhou valuation $\nu$ such that $\nu(\mathfrak{b}_{\bullet})=0$.
\end{Remark}

Let $\mathfrak{a}_{\bullet}$ be any graded sequence of ideals, and let $\mathfrak{q}$ be any nonzero ideal. Let $\mathfrak{b}_{\bullet}$ be the corresponding subadditive sequence of asymptotic multiplier ideals of $\mathfrak{a}_{\bullet}$, then we have $\nu(\mathfrak{a}_{\bullet})=\nu(\mathfrak{b}_{\bullet})$ and $\mathrm{lct}^{\mathfrak{q}}(\mathfrak{a}_{\bullet})=\mathrm{lct}^{\mathfrak{q}}(\mathfrak{b}_{\bullet})$. Thus, Proposition \ref{p:inf-b} implies the following

\begin{Proposition}
	\begin{equation}
		\nonumber
		\mathrm{lct}^{\mathfrak{q}}(\mathfrak{a}_{\bullet})=\inf_{\nu}\frac{\mathscr{A}(\nu)+\nu(\mathfrak{q})}{\nu(\mathfrak{a}_{\bullet})},
	\end{equation}
	where the infimum is over all Zhou valuations $\nu$. 
\end{Proposition}

\section{Zhou valuations and quasimonomial valuations}
In this section, we consider the relationship of Zhou valuations and quasimonomial valuations, and prove Remark \ref{rmk-ana.sing.is.quasimonomial}. Before the proof, we recall the concepts of quasimonomial valuation and log resolution for convenience.

\subsection{Quasimonomial valuation}\label{subsec-qm.val}

Recall the definition of quasimonomial valuation (see \cite{BFJ08,ELS03,FM05j,JM12,JM14,xu2019}). Let $X$ be a separated, regular, connected, excellent scheme over $\mathbb{Q}$. Let $\pi\colon Y\to X$ be a proper birational morphism, with $Y$ regular and connected, and $y=(y_1,\ldots,y_r)$ a system of algebraic coordinates at a point $\eta\in Y$ such that $Y$ is regular at $\eta$. Given $\alpha=(\alpha_1,\ldots,\alpha_r)\in\mathbb{R}_{\ge 0}^r\setminus\{0\}$. Denote
\[y^{\beta}=\prod_{i=1}^r y_i^{\beta_i}, \ \text{for} \ \beta=(\beta_1,\ldots,\beta_r)\in\mathbb{Z}_{\ge 0}^r,\]
and $\langle\alpha,\beta\rangle:=\sum_{i=1}^r\alpha_i\beta_i$ for $\alpha,\beta\in\mathbb{R}^r$. We can associate $\alpha$ with a valuation $\nu_{\alpha}$ as follows. For each $f\in\mathcal{O}_{Y,\eta}$, we write $f=\sum_{\beta\in\mathbb{Z}_{\ge 0}^r}c_{\beta}y^{\beta}$ with $c_{\beta}\in\widehat{\mathcal{O}_{Y,\eta}}$ either zero or a unit. Set
\[\nu_{\alpha}(f)=\min\big\{\langle \alpha,\beta\rangle\colon c_{\beta}\neq 0\big\}.\]
A \emph{quasimonomial valuation} is a valuation that can be written in the above form. It is well-known that a valuation $\nu$ on $X$ is quasimonomial if and only if the equality holds in the Abhyankar inequality
\[\mathrm{tr.deg\,}\nu+\mathrm{rat.rk\,}\nu\le \dim X\]
of this valuation (\cite{ELS03}, see also \cite[Proposition 3.7]{JM12}).

\subsection{Log resolution}

In addition, let us recall the log resolution of an ideal sheaf (see \cite{Hi64,Tem18}). Let $\mathfrak{a}\subseteq\mathcal{O}_X$ be a nonzero ideal sheaf on $X$ (where $X$ is a scheme in the setting of Section \ref{subsec-qm.val}, or simply a complex variety). A \emph{log resolution} of $\mathfrak{a}$ is a projective birational map $\pi\colon Y\to X$ such that
\[\pi^{-1}\mathfrak{a}\coloneqq  \mathfrak{a}\cdot\mathcal{O}_Y=\mathcal{O}_{Y}(-F),\]
where $F$ is an effective divisor on $Y$ such that $F+\mathrm{Exc}(\pi)$ has simple normal crossing support. Moreover, the relative canonical divisor
\[K_{Y/X}=K_Y-\pi^*K_X\]
is an effective divisor supported on the exceptional locus of $\pi$.

\subsection{Proof of Remark \ref{rmk-ana.sing.is.quasimonomial}}
\begin{proof}[Proof of Remark \ref{rmk-ana.sing.is.quasimonomial}]
	We may assume
	\[\Psi=c\log|\mathfrak{a}|,\]
	where $\mathfrak{a}$ is a nonzero ideal of $\mathcal{O}_o$ and $c\in\mathbb{R}_+$. Take a log resolution of $\mathfrak{a}$, denoted by $\pi\colon Y\to U$ ($U$ is an open neighborhood of $o$), such that $\mathfrak{a}\cdot\mathcal{O}_Y=\mathcal{O}_Y(-D)$, where $D=\sum_{1\le i\le s}a_iD_i$ is an effective divisor on $Y$ with simple normal crossing support. Thus,
	\[\Psi\circ\pi \sim c\sum_{1\le i\le s}a_i\log|u_i|,\]
	where $u_i$ are the local generators of $D_i$. As $o$ is an isolated singularity of $\Psi$, the pole set of $\Psi\circ \pi$ is compact (shrinking $U$ if necessary). Then we can verify that,
	\[\nu(g)=\sigma(\log|g|,\Psi; o)=\min_{p}\sigma(\log|g\circ\pi|,\Psi\circ\pi; p)=\min_{1\le i\le s}\frac{\mathrm{ord}_{D_i}(g)}{ca_i}, \ \forall (g,o)\in\mathcal{O}_o,\]
	where the minimum in the third term is over all singlar point $p$ of $\Psi\circ\pi$, and $\sigma(\cdot,\cdot;p)$ denotes the relative type near the point $p$. Since $\nu$ is a valuation by the assumption, there must be some $i_0$ such that $\nu=\mathrm{ord}_{D_{i_0}}(g)/ca_{i_0}$ (according to the following Lemma \ref{lem-min.val.is.val}), which shows that $\nu$ is a divisorial valuation.
\end{proof}

\begin{Lemma}\label{lem-min.val.is.val}
	Let $R$ be an integral domain, and $\{\nu_1,\ldots,\nu_s\}$ be a finite collection of homomorphisms of $R^*=R\setminus\{0\}$ into $\mathbb{R}_{\ge 0}$, i.e. for each $i\in\{1,\ldots,s\}$,
	\[\nu_i(ab)=\nu_i(a)+\nu_i(b), \ \forall a,b\in R^*.\]
	Set $\nu\colon R^*\to \mathbb{R}_{\ge 0}$ such that $\nu(a)=\min_{1\le i\le s}\nu_i(a)$ for every $a\in R^*$. If $\nu$ is also a homomorphism of $R^*$ into $\mathbb{R}_{\ge 0}$, then there must exist some $i_0\in\{1,\ldots,s\}$ such that $\nu=\nu_{i_0}$.
\end{Lemma}

\begin{proof}
	Suppose $\nu$ is a homomorphism of $R^*$ into $\mathbb{R}_{\ge 0}$. For every $a\in R^*$, denote
	\[I_{a}\coloneqq \{i\colon \nu(a)=\nu_{i}(a)\}\subset \{1,\ldots,s\}.\]

	First, we observe that for any finite subset $\{a_1,\ldots,a_r\}$ of $R^*$, $I_{a_1}\cap \cdots\cap I_{a_r}\neq\emptyset$. In fact, be the assumption, for each $k\in I_{a_1\cdots a_r}$, we have
	\begin{flalign*}
		\begin{split}
			\nu_k(a_1\cdots a_r)&=\nu(a_1\cdots a_r)\\
			&=\nu(a_1)+\cdots+\nu(a_r)\\
			&\le \nu_k(a_1)+\cdots+\nu_k(a_r)\\
			&=\nu_k(a_1\cdots a_r),
		\end{split}
	\end{flalign*}
	which shows $\nu_k(a_j)=\nu(a_j)$ for all $j\in\{1,\ldots,r\}$, and thus $k\in I_{a_1}\cap\cdots \cap I_{a_r}$. 

	Next, we prove that there exists $i_0\in\{1,\ldots,s\}$ such that $i_0\in I_a$ for every $a\in R^*$. If not, then for each $i\in\{1,\ldots,s\}$, there exists $b_i\in R^*$ such that $i\notin I_{b_i}$, and it follows that $I_{b_1}\cap\cdots\cap I_{b_s}=\emptyset$, which contradicts to the above observation.

	Consequently, we get $\nu(a)=\nu_{i_0}(a)$ for every $a\in R^*$, and the proof is completed.
\end{proof}

\section{Proofs of Theorem \ref{thm:compute2}, Theorem \ref{the:char} and Proposition \ref{c:exist}}

In this section, we prove Theorem \ref{thm:compute2}, Theorem \ref{the:char} and Proposition \ref{c:exist}.

First, we prove Theorem \ref{thm:compute2}.

\begin{proof}[Proof of Theorem \ref{thm:compute2}]
	Following from Lemma \ref{r:exists}, there is a local Zhou weight $\Phi_{o,\max}$ related to $|\mathfrak{q}|^2$ satisfying that 
	$$\Phi_{o,\max}\ge c_o^{\mathfrak{q}}(\varphi)\varphi$$ near $o$, and denote its corresponding Zhou valuation by $\nu$. Then we have $\mathscr{A}(\nu)+\nu(\mathfrak{q})=1$ and $\sigma(\varphi,\Phi_{o,\max})=\frac{1}{c_o^{\mathfrak{q}}(\varphi)}$. Note that $\mathrm{lct}^{\mathfrak{q}}(\mathfrak{b}^{\varphi}_{\bullet})=c_o^{\mathfrak{q}}(\varphi)$ (see \cite{JM14}). Using Lemma \ref{l:contro}, we get $$\mathrm{lct}^{\mathfrak{q}}(\mathfrak{b}^{\varphi}_{\bullet})\cdot	\nu(\mathfrak{b}^{\varphi}_{\bullet})=1=\mathscr{A}(\nu)+\nu(\mathfrak{q}).$$  Theorem \ref{thm:compute2} holds.
\end{proof}

Next, we prove Theorem \ref{the:char}.

\begin{proof}
	[Proof of Theorem \ref{the:char}]
	Firstly, we prove the necessity. Assume that there exists a local Zhou weight $\Phi_{o,\max}$ related $|\mathfrak{q}_0|^2$ near $o$ such that $$\nu(f,\Phi_{o,\max}):=\sigma(\log|f|,\Phi_{o,\max})=\nu(f)$$
	for any $(f,o)\in\mathcal{O}_o$.
	Note that there exist two positive constants $N_1$ and $N_2$ such that $$N_1\log|z|+O(1)\le \Phi_{o,\max}\le N_2\log|z|+O(1)$$
	near $o$, then we obtain that $\nu(f,\Phi_{o,\max})=\sup\{c\ge0:\log|f|\le c\Phi_{o,\max}+O(1)$ near $o\}=0$ if and only if $f(o)\not=0$, which shows that the statement $(1)$ holds.
	By the definition of $\mathfrak{a}_m^{\nu}$, we have 
	$$\log|\mathfrak{a}_m^{\nu}|\le m\Phi_{o,\max}+O(1)$$
	near $o$. Since $|\mathfrak{q}_0|^2e^{-2\Phi_{o,\max}}$ is not integrable near $o$, $|\mathfrak{q}_0|^2e^{-\frac{2}{m}\log|\mathfrak{a}_m^{\nu}|}$ is not integrable near $o$ for any $m$. 
	
	If there exists a valuation $\tilde\nu$ on $\mathcal{O}_o$ satisfying the statements $(1)$ and $(2)$, and 
	\begin{equation}
		\label{eq:1005c}\tilde\nu(f)\ge\nu(f)=\nu(f,\Phi_{o,\max})
	\end{equation}
	for any $(f,o)\in\mathcal{O}_o$. It follows from Lemma \ref{l:valuation--psh} that there exists a negative plurisubharmonic function $\varphi_{\tilde\nu}$ on $D$ such that for any $f\in\mathcal{O}(D)$ satisfying $|f|\le1$ on $D$,
	\begin{equation}
		\label{eq:1005e}\log|f|\le \tilde\nu(f)\varphi_{\nu}
	\end{equation}
	on $D$,
	$|\mathfrak{q}_0|^2e^{-2\varphi_{\tilde\nu}}$ is not integrable near $o$ and $|\mathfrak{q}_0|^2e^{-2\varphi_{\tilde\nu}}|z|^N$ is integrable near $o$ for large enough $N$, where $D$ is the unit ball in $\mathbb{C}^n$. 
	By the construction of $\varphi_{\tilde\nu}$ in the proof of Lemma \ref{l:valuation--psh}, we have
	\begin{equation}
		\label{eq:1005d}
		\varphi_{\tilde\nu}=\left(\sup\left\{\frac{\log|f(w)|}{\tilde\nu(f)} : f\in \mathcal{O}(D), \ \sup_D|f|\le 1, \ f(o)=0, \  f\not\equiv 0\right\}\right)^*.
	\end{equation}
	Lemma \ref{l:local-glabal} shows that there is a global Zhou weight $\Phi_{o,\max}^D$ on $D$ related to $|\mathfrak{q}_0|^2$ satisfying $$\Phi_{o,\max}^D=\Phi_{o,\max}+O(1)$$
	near $o$. 
	It follows from inequality \eqref{eq:1005c}, equality \eqref{eq:1005d} and Lemma \ref{cor-approximation} that
	\begin{equation}
		\nonumber
		\begin{split}
			\varphi_{\tilde\nu}&=\left(\sup\left\{\frac{\log|f(w)|}{\tilde\nu(f)} : f\in \mathcal{O}(D), \ \sup_D|f|\le 1, \ f(o)=0, \  f\not\equiv 0\right\}\right)^*\\
			\ge&\left(\sup\left\{\frac{\log|f(w)|}{\nu(f)} : f\in \mathcal{O}(D), \ \sup_D|f|\le 1, \ f(o)=0, \  f\not\equiv 0\right\}\right)^*\\
			=&\Phi_{o,\max}^D.
		\end{split}
	\end{equation}
	By the definition of the global Zhou weight $\Phi_{o,\max}^D$, we have $\varphi_{\tilde\nu}=\Phi_{o,\max}^D$ on $D$.
	Using inequality \eqref{eq:1005e}, we have 
	$$\log|f|\le \tilde\nu(f)\varphi_{\tilde\nu}=\tilde\nu(f)\Phi_{o,\max}^D,$$
	 for any $f\in\mathcal{O}(D)$ satisfying $|f|\le 1$ on $D$, which shows that  $\nu(f,\Phi_{o,\max})\ge \tilde\nu(f)$  for any  polynomial $f$ on $\mathbb{C}^n$. Using Lemma \ref{l:1029_1},  we have $$\nu(f,\Phi_{o,\max})\ge \tilde\nu(f)$$
	for any $(f,o)\in\mathcal{O}_o$,  which shows that 
	$$\nu(\cdot,\Phi_{o,\max})= \tilde\nu$$ by the assumption $\tilde\nu\ge \nu(\cdot,\Phi_{o,\max})$. Thus, the statement $(3)$ holds.

	Now, we prove the sufficiency. Assume that the three statements in Theorem \ref{the:char} hold. Lemma \ref{l:valuation--psh} shows that  there exists a plurisubharmonic function $\varphi_{\nu}$ on $D$ such that for any $f\in\mathcal{O}(D)$ satisfying $|f|\le 1$ on $D$,
	\begin{equation}
		\label{eq:1005a}\log|f|\le \nu(f)\varphi_{\nu}
	\end{equation}
	on $D$,
	$|\mathfrak{q}_0|^2e^{-2\varphi_{\nu}}$ is not integrable near $o$ and $|\mathfrak{q}_0|^2e^{-2\varphi_{\nu}}|z|^N$ is integrable near $o$ for large enough $N$, where $D$ is the unit ball in $\mathbb{C}^n$. 
	
	It follows from Lemma \ref{r:exists} that there exists a global Zhou weight $\Phi_{o,\max}^D$ related to $|\mathfrak{q}_0|^2$ such that 
	\begin{equation}
		\label{eq:1005b}\Phi_{o,\max}^D\ge \varphi_{\nu}
	\end{equation} 
	on $D$. Following from inequalities \eqref{eq:1005a} and \eqref{eq:1005b}, we obtain that  the Zhou valuation 
	$$\nu(f,\Phi_{o,\max}^D)\ge \nu(f)$$ for any  polynomial $f$ on $\mathbb{C}^n$. Using Lemma \ref{l:1029_1}, we obtain that
	$$\nu(f,\Phi_{o,\max}^D)\ge \nu(f)$$ for any $(f,o)\in\mathcal{O}_o$.
	 Note that the Zhou valuation $\nu(\cdot,\Phi_{o,\max}^D)$ satisfies the statements $(1)$ and $(2)$, then it follows from the statement $(3)$ that 
	$$\nu(\cdot,\Phi_{o,\max}^D)= \nu,$$
	i.e., $\nu$ is a Zhou valuation related to $|\mathfrak{q}_0|^2$.
	
	Thus, Theorem \ref{the:char} holds.
\end{proof}

Finally, we prove Proposition \ref{c:exist}.

\begin{proof}
	[Proof of Proposition \ref{c:exist}]
	Lemma \ref{l:valuation--psh} shows that there exists a negative plurisubharmonic function $\varphi_{\nu}$ on $D$ such that for any $f\in\mathcal{O}(D)$ satisfying $|f|\le 1$ on $D$,
	\begin{equation}
		\label{eq:1029b}
		\log|f|\le \nu(f)\varphi_{\nu}
	\end{equation}
    on $D$,
	$|\mathfrak{q}_0|^2e^{-2\varphi_{\nu}}$ is not integrable near $o$ and $|\mathfrak{q}_0|^2e^{-2\varphi_{\nu}}|z|^N$ is integrable near $o$ for large enough $N$, where $D$ is the unit ball in $\mathbb{C}^n$. It follows from Lemma \ref{r:exists} there exists a global Zhou weight $\Phi_{o,\max}^D$ on $D$ related to $|\mathfrak{q}_0|^2$ such that $$\Phi_{o,\max}^D\ge\varphi_{\nu}$$ on $D$. For any $f\in\mathcal{O}(D)$ satisfying $|f|\le 1$ on $D$, 
	$$\log|f|\le \nu(f)\varphi_{\nu}\le \nu(f)\Phi_{o,\max}^D.$$
	Combining Lemma \ref{l:1029_1}, we obtain that for any $(f,o)\in\mathcal{O}_o$,
	 the Zhou valuation $$\nu(f,\Phi_{o,\max}^D):=\sigma(\log|f|,\Phi_{o,\max})\ge\nu(f).$$ Thus, Proposition \ref{c:exist} holds.
\end{proof}

\section{Denseness of the Zhou valuations}
In this section, we discuss the denseness property of the subcone $\mathrm{ZVal}_o$ in $\mathrm{Val}_o$, and complete the proofs of Theorem \ref{thm.nu.inf.Zhou.val} and \ref{thm-ZVal.dense}.

\subsection{Some facts}
First we list some facts which will be used.

\begin{Lemma}
    Suppose there are two ideals $\mathfrak{a}\subset\mathfrak{a}'$, then for any nonzero ideal $\mathfrak{q}$, $\mathrm{lct}^{\mathfrak{q}}(\mathfrak{a})\le\mathrm{lct}^{\mathfrak{q}}(\mathfrak{a}')$.
    
    As a consequence, for two graded sequence of ideals $\mathfrak{a}_{\bullet},\mathfrak{a}'_{\bullet}$ satisfying $\mathfrak{a}_m\subset\mathfrak{a}'_m$ for sufficiently large $m$, it holds that $\mathrm{lct}^{\mathfrak{q}}(\mathfrak{a}_{\bullet})\le\mathrm{lct}^{\mathfrak{q}}(\mathfrak{a}'_{\bullet})$. Especially, for two valuations $\nu,\nu'$ satisfying $\nu\le\nu'$, it holds that $\mathrm{lct}^{\mathfrak{q}}(\mathfrak{a}^{\nu}_{\bullet})\le\mathrm{lct}^{\mathfrak{q}}(\mathfrak{a}^{\nu'}_{\bullet})$.
\end{Lemma}

\begin{proof}
    It is easy to be checked by the definitions.
\end{proof}

\begin{Lemma}\label{lem-existence.hat.nu.ge.nu}
    Let $\mathfrak{q}$ be a nonzero ideal of $\mathcal{O}_o$, $\nu\in\mathrm{Val}_o$, and $\mathfrak{a}_{m}^{\nu}:=\big\{h\in\mathcal{O}_o : \nu(h)\ge m\big\}$ be the graded sequence provided by $\nu$. If $\mathrm{lct}^{\mathfrak{q}}(\mathfrak{a}_{\bullet}^{\nu})\le 1$, then there exists a Zhou valuation $\tilde{\nu}$ related to $|\mathfrak{q}|^2$ such that $\tilde{\nu}\ge \nu$.
\end{Lemma}

\begin{proof}
    Since $\mathrm{lct}^{\mathfrak{q}}(\mathfrak{a}^{\nu}_{m})=\sup_{m\ge 1}m\cdot\mathrm{lct}^{\mathfrak{q}}(\mathfrak{a}_m)$, we obtain $\mathrm{lct}^{\mathfrak{q}}(\mathfrak{a}_m)\le 1/m$, or equivalently $\mathfrak{q}\not\subseteq \mathcal{I}\big(\frac{1}{m}\log|\mathfrak{a}_m^{\nu}|\big)_o$, for any $m\ge 1$. Now Proposition \ref{c:exist} indicates that there exists a Zhou valuation $\tilde{\nu}$ related to $|\mathfrak{q}|^2$ such that $\tilde{\nu}\ge \nu$ (where the first statement holds since $\nu$ is centered).
\end{proof}

\begin{Remark}\label{rem-mathcal.A.finite}
    Let $\nu\in \mathrm{Val}_o$, and $\mathfrak{a}_{m}^{\nu}:=\big\{h\in\mathcal{O}_o : \nu(h)\ge m\big\}$ be the graded sequence provided by $\nu$. Then the following statements are equivalent:

    (1) there exists $\hat{\nu}\in\mathrm{ZVal}_o$ satisfying $\hat{\nu}\ge \nu$;

    (2) for any nonzero ideal $\mathfrak{q}$ of $\mathcal{O}_o$, $\mathrm{lct}^{\mathfrak{q}}(\mathfrak{a}_{\bullet}^{\nu})<+\infty$;

    (3) there exists a nonzero ideal $\mathfrak{q}$ of $\mathcal{O}_o$, such that $\mathrm{lct}^{\mathfrak{q}}(\mathfrak{a}_{\bullet}^{\nu})<+\infty$.

    Moreover, if $\mathscr{A}(\nu)<+\infty$, then all the statements (1)(2)(3) hold.
\end{Remark}

\begin{proof}
    $(1)\Rightarrow (2)$: If $\hat{\nu}\in\mathrm{ZVal}_o$ satisfying $\hat{\nu}\ge \nu$, then
    \[\mathrm{lct}^{\mathfrak{q}}(\mathfrak{a}_{\bullet}^{\nu})\le \mathrm{lct}^{\mathfrak{q}}(\mathfrak{a}_{\bullet}^{\hat\nu})\le \hat\nu(\mathfrak{q})+1-\hat\nu(\mathfrak{q}_0)<+\infty,\]
    where it is assumed that $\hat\nu$ is a Zhou valuation related to $|\mathfrak{q}_0|^2$.
    
    $(2)\Rightarrow (3)$ is trivial. $(3)\Rightarrow (1)$: If $\mathrm{lct}^{\mathfrak{q}}(\mathfrak{a}_{\bullet}^{\nu})<+\infty$, then $\nu/M$ is a centered valuation satisfying $\mathrm{lct}^{\mathfrak{q}}(\mathfrak{a}_{\bullet}^{\nu/M})\le 1$ for sufficiently large $M>0$. Lemma \ref{lem-existence.hat.nu.ge.nu} implies that there exists a Zhou valuation $\tilde\nu\ge\nu/M$ related to $|\mathfrak{q}|^2$. Now $M\tilde\nu\in\mathrm{ZVal}_o$ satisfying $M\hat\nu\ge\nu$.

    If $\mathscr{A}(\nu)<+\infty$, then the statements (1)(2)(3) hold by the definition of $\mathscr{A}(\nu)$ and the equivalence of (1), (2) and (3).
\end{proof}

\subsection{Tian functions with respect to graded sequences}

Let $\nu: \mathcal{O}_o\to\overline{\mathbb{R}}_{\ge 0}$ be a centered valuation, and $\mathfrak{a}$ a nonzero ideal in $\mathcal{O}_o$. For any nonzero ideal $\mathfrak{q}$ of $\mathcal{O}_o$, and $t>0$, denote
\[\mathrm{lct}^{\mathfrak{q}^t}(\mathfrak{a}):=\sup\big\{c\ge 0 : |\mathfrak{q}|^{2t}e^{-2c\log|\mathfrak{a}|} \ \text{is locally integrable near} \ o\big\}.\]
We have
\begin{Lemma}\label{lem-lct.qt.incr.concave}
    Suppose $\mathrm{lct}^{\mathfrak{q}}(\mathfrak{a})<+\infty$. Then the function
    \[\mathrm{Tn}(t;\mathfrak{a},\mathfrak{q}):=\mathrm{lct}^{\mathfrak{q}^t}(\mathfrak{a})\]
    is non-decreasing and concave with respect to $t\in (0,+\infty)$.
\end{Lemma}

\begin{proof}
    First, we show that for any $s<t\in (0,+\infty)$,
    \[\mathrm{lct}^{\mathfrak{q}^s}(\mathfrak{a})<+\infty \Leftrightarrow \mathrm{lct}^{\mathfrak{q}^t}(\mathfrak{a})<+\infty.\]
    It is clear that $\mathrm{lct}^{\mathfrak{q}^s}(\mathfrak{a})\le\mathrm{lct}^{\mathfrak{q}^t}(\mathfrak{a})$, which implies the ``$\Leftarrow$'' and that $\mathrm{Tn}( \ \cdot \ ;\mathfrak{a},\mathfrak{q})$ is increasing. Now we assume $\mathrm{lct}^{\mathfrak{q}^s}(\mathfrak{a})<+\infty$. Let $U$ be a bounded open neighborhood of $o$ in $\mathbb{C}^n$ and $c<+\infty$ satisfying that
    \[\int_{U}|\mathfrak{q}|^{2s}e^{-2c\log|\mathfrak{a}|}=+\infty.\]
    Then according to H\"{o}lder's inequality, we have
    \[\left(\int_U|\mathfrak{q}|^{2t}e^{-2tc\log|\mathfrak{a}|/s}\right)^{s/t}\left(\int_U 1\right)^{1-s/t}\ge\int_{U}|\mathfrak{q}|^{2s}e^{-2c\log|\mathfrak{a}|}=+\infty,\]
    which shows $\mathrm{lct}^{\mathfrak{q}^t}(\mathfrak{a})\le tc/s<+\infty$. 

    H\"{o}lder's inequality also gives the concavity of $\mathrm{Tn}(t;\mathfrak{a},\mathfrak{q})$. For any $t_1<t<t_2\in (0,+\infty)$, and $c_1,c_2\in [0,+\infty)$, H\"{o}lder's inequality implies
    \[\int_V|\mathfrak{q}|^{2t}e^{-2c\log|\mathfrak{a}|}\le\left(\int_V|\mathfrak{q}|^{2t_1}e^{-2c_1\log|\mathfrak{a}|}\right)^{1/\alpha}\left(\int_V |\mathfrak{q}|^{2t_2}e^{-2c_2\log|\mathfrak{a}|}\right)^{1-1/\alpha}\]
    for some bounded open neighborhood $V$ of $o$, where
    \[\alpha=\frac{t_2-t_1}{t_2-t}, \ c=\frac{t_2-t}{t_2-t_1}c_1+\frac{t-t_1}{t_2-t_1}c_2.\]
    We can deduce that
    \[\mathrm{Tn}(t;\mathfrak{a},\mathfrak{q})\ge\frac{t_2-t}{t_2-t_1}\mathrm{Tn}(t_1;\mathfrak{a},\mathfrak{q})+\frac{t-t_1}{t_2-t_1}\mathrm{Tn}(t_2;\mathfrak{a},\mathfrak{q}), \ \forall t_1<t<t_2,\]
    which shows the concavity of $\mathrm{Tn}(t;\mathfrak{a},\mathfrak{q})$.
\end{proof}

Let $\nu\in\mathrm{Val}_o$, and
\[\mathfrak{a}_{m}^{\nu}:=\big\{h\in\mathcal{O}_o : \nu(h)\ge m\big\},\]
where $m\in\mathbb{N}_+$. Then $\mathfrak{a}_{\bullet}^{\nu}$ is a graded sequence of ideals of $\mathcal{O}_o$. One can obtain
\begin{Lemma}\label{lem-lct.qt.avbullet}
    Suppose a nonzero ideal $\mathfrak{q}$ and a graded sequence of ideals $\mathfrak{a}_{\bullet}$ satisfy $\mathrm{lct}^{\mathfrak{q}}(\mathfrak{a}_{\bullet})<+\infty$. Then the function
    \[\mathrm{Tn}(t;\mathfrak{a}_{\bullet},\mathfrak{q}):=\mathrm{lct}^{\mathfrak{q}^t}(\mathfrak{a}_{\bullet})\]
    is increasing and concave with respect to $t\in (0,+\infty)$.

    Moreover, if $\mathfrak{a}_{\bullet}=\mathfrak{a}_{\bullet}^{\nu}$ for some $\nu\in\mathrm{Val}_o$ with $\mathscr{A}(\nu)<+\infty$, then
    \[\lim_{t\to+\infty}\frac{\mathrm{Tn}(t;\mathfrak{a}_{\bullet}^{\nu},\mathfrak{q})}{t}=\nu(\mathfrak{q}).\]
\end{Lemma}

\begin{proof}
    Recall that
    \[\mathrm{lct}^{\mathfrak{q}^t}(\mathfrak{a}_{\bullet})=\lim_{m\to\infty}m\cdot\mathrm{lct}^{\mathfrak{q}^t}(\mathfrak{a}_{m})=\sup_{m}m\cdot\mathrm{lct}^{\mathfrak{q}^t}(\mathfrak{a}_{m}).\]
    $\mathrm{lct}^{\mathfrak{q}}(\mathfrak{a}_{\bullet})<+\infty$ yields that $\mathrm{lct}^{\mathfrak{q}}(\mathfrak{a}_{m})<+\infty$ for any $m\in\mathbb{N}_+$. Lemma \ref{lem-lct.qt.incr.concave} shows that $\mathrm{lct}^{\mathfrak{q}^t}(\mathfrak{a}_{m})$ is increasing and concave with respect to $t\in (0,+\infty)$ for any fixed $m$. Then it follows that $\mathrm{lct}^{\mathfrak{q}^t}(\mathfrak{a}_{\bullet})$ is also increasing and concave. Especially, $\mathrm{lct}^{\mathfrak{q}^t}(\mathfrak{a}_{\bullet})<+\infty$ for any $t\in (0,+\infty)$.

    Next, let $\nu\in\mathrm{Val}_o$ with $\mathscr{A}(\nu)<+\infty$. It follows that $\mathrm{lct}^{\mathfrak{q}}(\mathfrak{a}^{\nu}_{\bullet})<+\infty$. Then $\mathrm{Tn}(t;\mathfrak{a}_{\bullet}^{\nu},\mathfrak{q})$ is concave, and thus the limit $\lim\limits_{t\to+\infty}\mathrm{Tn}(t;\mathfrak{a}_{\bullet}^{\nu},\mathfrak{q})/t$ exists. Now we only need to prove $\lim\limits_{k\to+\infty}\mathrm{Tn}(k;\mathfrak{a}_{\bullet}^{\nu},\mathfrak{q})/k=\nu(\mathfrak{q})$ for $k\in\mathbb{N}_+$. Note that
    \[\mathrm{lct}^{\mathfrak{q}^k}(\mathfrak{a}^{\nu}_{\bullet})-k\nu(\mathfrak{q})=\mathrm{lct}^{\mathfrak{q}^k}(\mathfrak{a}^{\nu}_{\bullet})-\nu(\mathfrak{q}^k)\le\mathscr{A}(\nu)<+\infty, \ \forall k\in\mathbb{N}_+.\]
    Then we get 
    \[\lim_{k\to\infty}\frac{\mathrm{lct}^{\mathfrak{q}^k}(\mathfrak{a}^{\nu}_{\bullet})}{k}\le\nu(\mathfrak{q}).\]
    On the other hand, assume $\nu(\mathfrak{q})>0$. Let $a/b\in (0,\nu(\mathfrak{q})]\cap\mathbb{Q}$, where $a,b\in\mathbb{N}_+$ with $(a,b)=1$, and $\mathfrak{q}_{\bullet}=\{\mathfrak{q}_m\}=\{\mathfrak{q}^{mb}\}_{m=1}^{\infty}$ be a graded sequence of ideals. Then $\mathfrak{q}_m\subset\mathfrak{a}_{ma}^{\nu}$ for any $m$, and $\mathrm{lct}^{\mathfrak{q}^k}(\mathfrak{q}_{\bullet})\ge k/b$ for any $k\in\mathbb{N}_+$. It follows that
    \[\lim_{k\to\infty}\frac{\mathrm{lct}^{\mathfrak{q}^k}(\mathfrak{a}^{\nu}_{\bullet})}{k}\ge a\limsup_{k\to\infty}\frac{\mathrm{lct}^{\mathfrak{q}^k}(\mathfrak{q}_{\bullet})}{k}\ge \frac{a}{b}.\]
    Thus, $\lim\limits_{k\to+\infty}\mathrm{Tn}(k;\mathfrak{a}_{\bullet}^{\nu},\mathfrak{q})/k=\nu(\mathfrak{q})$.

    The proof is completed.
\end{proof}

\subsection{Proofs of Theorem \ref{thm.nu.inf.Zhou.val} and \ref{thm-ZVal.dense}}

Now we prove Theorem \ref{thm.nu.inf.Zhou.val} and \ref{thm-ZVal.dense}.

\begin{proof}[Proof of Theorem \ref{thm.nu.inf.Zhou.val}]
    Fix any nonzero ideal $\mathfrak{q}$ of $\mathcal{O}_o$. Since $\mathscr{A}(\nu)<+\infty$, we have $\mathrm{lct}^{\mathfrak{q}}(\mathfrak{a}_{\bullet}^{\nu})<+\infty$. For any $k\in\mathbb{N}_+$, denote
\[C^{\mathfrak{q}}_k:=\mathrm{lct}^{\mathfrak{q}^k}(\mathfrak{a}^{\nu}_{\bullet})<+\infty.\]
Then Lemma \ref{lem-existence.hat.nu.ge.nu} verifies that there exists a Zhou valuation $\tilde{\nu}_k$ related to $|\mathfrak{q}^k|^2$ such that $C^{\mathfrak{q}}_k\tilde{\nu}_k\ge \nu$, and $\tilde{\nu}_k(\mathfrak{q})\le 1/k$. Set
\[\hat{\nu}_k:=C^{\mathfrak{q}}_k\tilde{\nu}_k\in\mathrm{ZVal}_o.\]
Then $\hat{\nu}_k\ge \nu$, and
\[\nu(\mathfrak{q})\le\hat{\nu}_k(\mathfrak{q})\le \frac{C_k^{\mathfrak{q}}}{k}.\]
Lemma \ref{lem-lct.qt.avbullet} shows that $\lim_{k\to \infty}C_{k}^{\mathfrak{q}}/k=\nu(\mathfrak{q})$. Then we complete the proof of the theorem. 
\end{proof}

\begin{proof}[Proof of Theorem \ref{thm-ZVal.dense}]
    We need to prove that for any finite set $\{\mathfrak{q}_1,\ldots,\mathfrak{q}_r\}$ of ideals of $\mathcal{O}_o$, any $\nu\in\mathrm{Val}_o$, and any $n\in\mathbb{N}_+$, there exists $\hat{\nu}\in\mathrm{ZVal}_o$, such that
    \[\max_{1\le i \le r}|\hat\nu(\mathfrak{q}_i)-\nu(\mathfrak{q}_i)|<\frac{1}{n}.\]

    Without loss of generality, we assume that $\mathfrak{q}_1,\ldots,\mathfrak{q}_r$ are all nontrivial ideals. Set $\mathfrak{q}:=\mathfrak{q}_1\cdots\mathfrak{q}_r$. According to Theorem \ref{thm.nu.inf.Zhou.val}, we can find some $\hat{\nu}\in\mathrm{ZVal}_o$ such that $\hat\nu\ge\nu$, and
    \[\hat{\nu}(\mathfrak{q})<\nu(\mathfrak{q})+\varepsilon,\]
    for any $\varepsilon>0$ sufficiently small. Set
    \[\varepsilon=\frac{1}{n}\cdot\min_{1\le i\le r}\left(\prod_{1\le j\le r, \ j\ne i}\nu(\mathfrak{q}_j)\right).\]
    Then we have that
    \begin{flalign*}
        \begin{split}
            \prod_{1\le j\le r}\nu(\mathfrak{q}_j)&\le\prod_{1\le j\le r}\hat\nu(\mathfrak{q}_j)=\hat\nu(\mathfrak{q})<\nu(\mathfrak{q})+\varepsilon\\
            &=\nu(\mathfrak{q})+\frac{1}{n}\cdot\min_{1\le i\le r}\left(\prod_{1\le j\le r, \ j\ne i}\nu(\mathfrak{q}_j)\right)\\
            &=\min_{1\le i\le r}\left(\left(\frac{1}{n}+\nu(\mathfrak{q}_i)\right)\cdot\prod_{1\le j\le r, \ j\ne i}\nu(\mathfrak{q}_j)\right)
        \end{split}
    \end{flalign*}
    yields
    \[\nu(\mathfrak{q}_i)\le\hat\nu(\mathfrak{q}_i)<\nu(\mathfrak{q}_i)+\frac{1}{n}, \ \forall 1\le i\le r,\]
    which gives the desired $\hat\nu\in\mathrm{ZVal}_o$.
\end{proof}

\vspace{.1in} {\em Acknowledgements}.
The authors would like to thank Professor Mattias Jonsson for helpful discussions. The authors would also like to thank Xun Sun, and Dr. Zhitong Mi for checking the manuscript.
The second named author was supported by National Key R\&D Program of China 2021YFA1003100, NSFC-11825101 and NSFC-12425101. The third named author was supported by China Postdoctoral Science Foundation BX20230402 and 2023M743719.

\bibliographystyle{references}
\bibliography{xbib}

\begin{thebibliography}{100}
	
\bibitem{BGMY-valuation}S.J. Bao, Q.A. Guan, Z.T. Mi and Z. Yuan, Tame maximal weights, relative types and valuations, \href{https://arxiv.org/abs/2310.00368v2}{arXiv:2310.00368v2}.

\bibitem{berndtsson13}B. Berndtsson, The openness conjecture for plurisubharmonic functions, \href{https://arxiv.org/abs/1305.5781}{arXiv:1305.5781}.

\bibitem{BFJ08}S. Boucksom, C. Favre and M. Jonsson,
Valuations and plurisubharmonic singularities.
Publ. Res. Inst. Math. Sci. 44 (2008), no. 2, 449-494.


\bibitem{demailly-note2000}J.-P. Demailly, Multiplier ideal sheaves and analytic methods in algebraic geometry. School on Vanishing Theorems and Effective Results in Algebraic Geometry (Trieste, 2000), 1--148, ICTP Lect. Notes, 6, Abdus Salam Int. Cent. Theoret. Phys., Trieste, 2001.


\bibitem{demailly2010}J.-P. Demailly, Analytic Methods in Algebraic Geometry, Higher Education Press, Beijing, 2010.
	
\bibitem{demailly-book}J.-P. Demailly, Complex analytic and differential geometry, electronically accessible
at \href{https://www-fourier.ujf-grenoble.fr/~demailly/manuscripts/agbook.pdf}{https://www-fourier.ujf-grenoble.fr/\textasciitilde demailly/manuscripts/agbook.pdf}.
	
\bibitem{DEL00}J.-P. Demailly, L. Ein and R. Lazarsfeld, A subadditivity property of multiplier ideals, Michigan Math. J. 48 (2000), 137-156.


\bibitem{D-K01}J.-P. Demailly and J. Koll\'{a}r,
Semi-continuity of complex singularity exponents and K\"{a}hler-Einstein metrics on Fano orbifolds.
Ann. Sci. \'{E}cole Norm. Sup. (4) 34 (2001), no. 4, 525-556.

\bibitem{ELS03}
L. Ein, R. Lazarsfeld and K. E. Smith, Uniform approximation of Abhyankar valuations in smooth function fields, Amer. J. Math. 125 (2003), p. 409-440.

\bibitem{FM05j}C. Favre and M. Jonsson, Valuations and multiplier ideals, J. Amer. Math. Soc. 18 (2005), no. 3, 655-684.

\bibitem{G-R}H. Grauert and R. Remmert, Coherent analytic sheaves, Grundlehren der mathematischen Wissenchaften, 265, Springer-Verlag, Berlin, 1984.


\bibitem{guan-effect}Q.A. Guan, A sharp effectiveness result of Demailly's strong openness conjecture, Adv. Math. 348 (2019), 51--80.

\bibitem{GZopen-c}Q.A. Guan and X.Y. Zhou, A proof of Demailly's strong openness conjecture,
Ann. of Math. (2) 182 (2015), no. 2, 605--616. See also arXiv:1311.3781.

\bibitem{GZopen-effect}Q.A. Guan and X.Y. Zhou, Effectiveness of Demailly's strong openness conjecture and related problems, Invent. Math. 202 (2015), no. 2, 635-676.

\bibitem{Hi64}H. Hironaka. Resolution of singularities of an algebraic variety over a field of characteristic zero. I, II. Ann. Math. 79 (1964), 109-203; ibid. 205-326.
	
\bibitem{JM12}M. Jonsson and M. Musta\c{t}\u{a}, Valuations and asymptotic invariants for sequences of ideals, Annales de l'Institut Fourier A. 2012, vol. 62, no.6, pp. 2145-2209.
	
\bibitem{JM14}M. Jonsson and M. Musta\c{t}\u{a}, An algebraic approach to the openness conjecture of Demailly and Koll\'{a}r, J. Inst. Math. Jussieu 13 (2014), no. 1, 119-144.

\bibitem{LarII04}R. Lazarsfeld, Positivity in Algebraic Geometry. II. Positivity for vector bundles, and multiplier ideals. Ergebnisse der Mathematik und ihrer Grenzgebiete. 3. Folge. A Series of Modern Surveys in Mathematics [Results in Mathematics and Related Areas. 3rd Series. A Series of Modern Surveys in Mathematics], 49. Springer-Verlag, Berlin, 2004.

\bibitem{Nadel90}A. Nadel, Multiplier ideal sheaves and K\"{a}hler-Einstein metrics of positive scalar curvature.
Ann. of Math. (2) 132 (1990), no. 3, 549-596.
	
\bibitem{Rash06}A. Rashkovskii,
Relative types and extremal problems for plurisubharmonic functions, Int. Math. Res. Not. 2006, Art. ID 76283, 26 pp.
	
\bibitem{skoda1972}H. Skoda, Sous-ensembles analytiques d'ordre fini ou infini dans $\mathbb{C}^n$, Bull. Soc. Math. France 100 (1972) 353-408.

\bibitem{Tem18}
Temkin, M. Functorial desingularization over $\mathbf{Q}$: boundaries and the embedded case. Isr. J. Math. 224, 455-504 (2018).
	
\bibitem{xu2019}C.Y. Xu, A minimizing valuation is quasi-monomial, Ann. of Math. (2) 191 (2020), no. 3, 1003-1030.
\end{thebibliography}

\end{document}